\documentclass{amsart}

\usepackage{amssymb}
\usepackage{mathrsfs}
\usepackage{amsmath}

\newtheorem{theorem}{Theorem}[section]
\newtheorem{lemma}[theorem]{Lemma}
\newtheorem{proposition}[theorem]{Proposition}
\newtheorem{corollary}[theorem]{Corollary}
\theoremstyle{definition}

\theoremstyle{remark}
\newtheorem{remark}[theorem]{Remark}

\numberwithin{equation}{section}

\begin{document}

\title{K{\"a}hler Einstein manifolds and the Calabi curvature operator}


\author{Zhi-Lin Dai}
\address{School of Science, East China University of Technology, Fuzhou, 344000, Jiangxi, China}
\email{202361014@ecut.edu.cn}

\author{Hai-Ping Fu$^*$}
\address{Department of Mathematics,  Nanchang University, Nanchang 330031, P. R. China}
\email{mathfu@126.com}
\thanks{* Corresponding author}
\thanks{Supported in part by National Natural Science Foundations of China \#12461008 and 12271069, Jiangxi
Province Natural Science Foundation of China \#20202ACB201001, Jiangxi Province Graduate Student Innovation Special Fund Project \#YC2025-B037}

\author{Yao Lu}
\address{Department of Mathematics,  Nanchang University, Nanchang 330031, P. R. China}
\email{luyao@email.ncu.edu.cn}




\begin{abstract}
In this paper, we study the Calabi curvature operator on K{\"a}hler manifolds. First, we prove that if the Calabi curvature operator on K{\"a}hler manifolds satisfies $\frac{{n\left(n + 1\right)}}{2}$-positive (nonnegative), $\frac{{n + 1}}{2}$-positive (nonnegative), and $\left( n-1 \right)$-positive (nonnegative), then the scalar curvature, Ricci curvature, and orthogonal Ricci curvature are positive (nonnegative), respectively. Second, we show that any compact K{\"a}hler Einstein manifold satisfying the condition
$$\lambda_1+\dots+\lambda_{\alpha}\ge -{\alpha}\theta(n,{\alpha})\bar\lambda,\; {\alpha}\le \frac{n}{2}$$
must have nonnegative constant holomorphic sectional curvature.
\end{abstract}

\maketitle
\section{introduction}
In a K{\"a}hler manifold $\left( {M,g,J} \right)$ of real dimension $2n$, let $V=T_{p} M$ denote the tangent space at a point $p\in M$. On the vector space $V$, one may define several types of curvature operators. For example, the curvature operator of the first kind \cite{tachibana_theorem_1974} (often simply called the curvature operator) is defined as
\begin{align*}
  \hat R:&{ \wedge ^2}\left( V \right) \to { \wedge ^2}\left( V \right) \hfill \\
  &{e_i} \wedge {e_j} \mapsto \frac{1}{2}\sum\limits_{k,l = 1}^{2n} {{R_{ijkl}}{e_k} \wedge {e_l}},
\end{align*}
where $\left\{e_i \right\}_{i=1}^{2n}$ is a local orthonormal frame. By restricting the curvature operator $\hat R$ to $\mathfrak{u}\left( V \right)  \subset \mathfrak{so}\left( {V} \right)  \cong { \wedge ^2}\left( V \right)$, one obtains the K{\"a}hler curvature operator. In \cite{petersen2021vanishing}, Petersen and Wink showed how the K{\"a}hler curvature operator affects harmonic $\left( p,q \right)$-forms, and obtained a rigidity conclusion on K{\"a}hler-Einstein manifolds.
Similarly, the curvature operator of the second kind is defined as
\begin{align*}
  \mathring{R}:&S_0^2\left( V \right) \to S_0^2\left( V \right) \hfill \\
  &{e_i} \odot {e_l} \mapsto \sum\limits_{j,k = 1}^{2n} {{R_{ijkl}}{e_j} \odot {e_k}},
\end{align*}
where $S_0^2$ denotes the space of traceless symmetric $2$-tensors on $V$, and $\odot$ denotes the symmetric tensor product.
In \cite{li2023kahler}, it was shown that a real $2n$ dimensional K{\"a}hler manifolds whose curvature operator of the second kind is $\frac{3}{2}\left( {{n^2} - 1} \right)$-nonnegative must have nonnegative constant holomorphic sectional curvature. Furthermore, if such a manifold is closed and its curvature operator of the second kind is $\frac{{3{n^3} - n + 2}}{{2n}}$-positive , then it is biholomorphic to $\mathbb{CP}^n$.

In the present paper, we focus on the Calabi curvature operator, defined as
\begin{align*}
\mathcal{C} :{ \odot ^2}{V^{1,0}}& \to { \odot ^2}{V^{1,0}},\\
{Z_A} \odot {Z_B}& \mapsto \sum\limits_{C,D = 1}^n {{R_{A\bar C\bar DB}}{Z_C} \odot {Z_D}},
\end{align*}
where $V^{1,0}$ denotes the $(1,0)$-part of the complexified tangent space.
Therefore, for any $S \in { \odot ^2}{V^{1,0}}$ expressed as $$S = \sum\limits_{A,B = 1}^n {{S_{AB}}{Z_A} \otimes {Z_B}},$$ we have
\begin{equation}
\label{equation1.1}
\mathcal{C}\left( S \right) = \sum\limits_{A,B,C,D = 1}^n {{S_{AB}}{R_{A\bar C\bar DB}}{Z_C} \otimes {Z_D}} .
\end{equation}
The Calabi operator was first introduced in \cite{calabi1960compact}, where Calabi and Vesentini computed its eigenvalues and multiplicities on several classes of K$\ddot{a}$hler manifolds.  Ogiue and Tachibana \cite{ogiue1980kaehler} showed the K{\"a}hler manifolds with positive Calabi operator have the rational cohomology of $\mathbb{CP}^n$, \cite{broder2025vanishing,wang_weitzenb$backslash$_2025} weaken the condition to $\frac{n}{2}$-positive and resulted in the same conclusion as above, and they also obtained some other interesting conclusions in their paper. The Calabi operator is also studied in these articles \cite{borel1960curvature,sitaramayya1973curvature}.

In this paper, let $\left\{ {{\lambda _\beta}} \right\}_{\beta = 1}^{\frac{{n\left( {n + 1} \right)}}{2}}$ denote the increasingly ordered eigenvalues of Calabi operator $\mathcal{C}$ on an $n$-dimensional K{\"a}hler manifold $\left( {M,g,J} \right)$. For any real number $\alpha >0$, we define the sum of the first $\alpha $ eigenvalues of $\mathcal{C}$ as
$${\lambda _1} +  \cdots  + {\lambda _\alpha }: = {\lambda _1} +  \cdots  + {\lambda _{\left[ \alpha  \right]}} + \left( {\alpha  - \left[ \alpha  \right]} \right){\lambda _{\left[ \alpha  \right] + 1}},$$
where $[\alpha ]$ is the integer part of $\alpha $.
We say that $M$ has $\alpha $-positive(resp., $\alpha $-nonnegative) Calabi operator if
$${\lambda _1} +  \cdots  + {\lambda _\alpha } > \left(  \ge  \right)0$$
holds at every point $p \in M$.

We first state a result concerning the curvature tensor.
\begin{theorem}
\label{theorem1.1}
Let $\left( {M,g,J} \right)$ be a K{\"a}hler manifold of complex dimension $n$.
\begin{enumerate}
\item If $M$ has $\frac{{n\left(n + 1\right)}}{2}$-positive (resp., nonnegative) Calabi curvature operator, then $M$ has positive (resp., nonnegative) scalar curvature.
\item If $M$ has $\frac{{n + 1}}{2}$-positive (resp., nonnegative) Calabi curvature operator, then $M$ has positive (resp., nonnegative) Ricci curvature.
\item If $M$ has $\left(n-1\right)$-positive (resp., nonnegative) Calabi curvature operator, then $M$ has positive (resp., nonnegative) orthogonal Ricci curvature.
\item If $M$ has $n$-positive (resp., nonnegative) Calabi curvature operator, then at every point it satisfies
$$2{Ric\left( {{e_b},{e_b}} \right) - R\left( {{e_b},J{e_b},{e_b},J{e_b}} \right)}> \left( \ge \right) 0.$$

\item If $M$ has $\alpha $-positive Calabi curvature operator for some $\alpha  \le n-1$, then $M$ is irreducible.
\end{enumerate}
\end{theorem}
By Bochner's classical result \cite{bochner1946vector}, positive Ricci curvature implies the vanishing of holomorphic forms. Hence we obtain:
\begin{corollary}
\label{corollary1.2}
 If an $n$-dimensional K{\"a}hler manifold admits $\frac{{n + 1}}{2}$-positive Calabi operator, then its Hodge numbers satisfy ${h^{m,0}}=0$ for all $1 \le m \le n$.
\end{corollary}
In \cite{ni2018comparison}, Ni and Zheng showed that positive orthogonal Ricci curvature forces the vanishing of ${h^{n - 1,0}}$ and ${h^{2,0}}$ on compact K{\"a}hler manifolds. Therefore:
\begin{corollary}
\label{corollary1.3}
Let $M$ be a compact K{\"a}hler manifold of dimension $n\geq2$. If its Calabi operator is $\left(n-1\right)$-positive, then the Hodge numbers ${h^{n - 1,0}}$ and ${h^{2,0}}$ both vanish.
\end{corollary}
Finally, we establish a rigidity result for compact K{\"a}hler-Einstein manifolds. Note that K{\"a}hler manifolds with $(n-1)$-positive Calabi curvature operator and harmonic Weyl tensor must be K{\"a}hler-Einstein by (5) of Theorem~\ref{theorem1.1}.
\begin{theorem}
\label{theorem1.4}
Let $\left( {M,g,J} \right)$ be an $n$-dimensional compact K{\"a}hler-Einstein manifold. If the smallest $\beta $ eigenvalues of Calabi curvature operator satisfy
\begin{equation}
\label{equation1.2}
{\lambda _1} +  \cdots  + {\lambda _\beta } \ge  - \beta \theta \bar \lambda, \ \ \beta \le \frac{n}{2}
\end{equation}
where
$$\theta \left( {n,\beta } \right) = \frac{{ - 2N \beta - \left( {n - 1} \right)N + \left( {n + 3} \right)\beta }}{{2N \beta  + 2nN - 2\left( {n + 2} \right)\beta}},\ \ \ N = \frac{{n\left( {n + 1} \right)}}{2},$$
then $M$ have nonnegative constant holomorphic sectional curvature.
\end{theorem}
When $\beta =1$, results from \cite{siu1980compact} and \cite{matsushima1957structure} imply that a K{\"a}hler-Einstein manifold with a positive Calabi operator must have positive constant holomorphic curvature. Theorem~\ref{theorem1.4}, however, allows negative eigenvalues when $\beta >1$. For instance, taking $n=12$ and $\beta =6$ gives $\theta \left( {12,6} \right) = - \frac{71}{110}$; one may then choose eigenvalues such as  ${\lambda _1} =  - \frac{51}{40} \bar \lambda,{\lambda _2} = \cdots= {\lambda _N} = \frac{453}{440} \bar \lambda$, and condition (1.2) is still satisfied.

\section{Preliminaries}
\subsection{Notation and Conventions}
 In the following, $\left( V,g,J \right)$ denotes a $2n$-dimensional Euclidean vector space, where $J$ is a complex structure on $V$, i.e., for $u,v \in V$, it satisfies
$$J^2=-id,g\left( {u,v} \right) = g\left( {Ju,Jv} \right).$$
Under these conditions, there  exists an orthonormal basis $\left\{ {{e_i}} \right\}_{i = 1}^{2n}$ of $V$ such that
\begin{equation}
\label{equation2.1}
J{e_a} = {e_{a + n}},J{e_{a + n}} =  - {e_a},a = 1,2, \cdots ,n.
\end{equation}
Let $V^*$ be the dual vector space of $V$. We identify a vector $u\in V$ with its dual $u^* \in V^*$ via the inner product $g$, such that for all $v\in V$,
$$g\left( {u,v} \right)={u^*}\left( {v} \right).$$
Consequently, we can identify  $End \left( V \right)$ with ${ \otimes ^2}{V } = {\wedge} ^{2} V \oplus {\odot} ^{2} V$ through the following correspondences:
$$\left( {u \otimes v} \right)\left( {w} \right) = g\left( {u,w} \right)v,$$
and
$$\left( {u \wedge v} \right)\left( {w} \right) = g\left( {u,w} \right)v - g\left( {v,w} \right)u,  \left( {u \odot v} \right)\left( {w} \right) = g\left( {u,w} \right)v + g\left( {v,w} \right)u.$$
For a tensor $T \in \otimes ^k V$, its (squared) tensor norm is defined as
$${\left| T \right|^2} = \sum\limits_{{i_1} \cdots {i_k}}^{} {{T_{{i_1} \cdots {i_k}}}{T_{{i_1} \cdots {i_k}}}}  ,$$
where ${T_{{i_1} \cdots {i_k}}} = T\left( {{e_{{i_1}}}, \cdots ,{e_{{i_k}}}} \right)$ are the components of $T$ with respect to the orthonormal basis $\left\{ {{e_i}} \right\}_{i = 1}^{2n}$ of $V$.

Now, for a symmetric tensor $S \in {\odot} ^2 V$, we can define a linear operator
\begin{align*}
  S:{ \otimes ^k}V &\to { \otimes ^k}V, \\
    T &\mapsto S(T),
\end{align*}
by
$$\left[ {S(T)} \right]\left( {{e_{{i_1}}}, \cdots ,{e_{{i_k}}}} \right): = \sum\limits_{r = 1}^k {T\left( {{e_{{i_1}}}, \cdots ,S\left( {{e_{{i_r}}}} \right) \cdots ,{e_{{i_k}}}} \right)} .$$

\subsection{Complexification}
Let ${V^{\mathbb{C}}} = V{ \otimes _{\mathbb{R}}{\mathbb{C}}}$ be the complexification of $\left( V,g,J \right)$, and extend all $\mathbb{R}$-linear maps on $V$ to be $\mathbb{C}$-linear on  ${V^{\mathbb{C}}}$. Then ${V^{\mathbb{C}}}$ decomposes as
$${V^{\mathbb{C}}} = {V^{1,0}} \oplus {V^{0,1}},$$
where ${V^{1,0}}$ and ${V^{0,1}}$ are the eigenspaces of $J$ corresponding to eigenvalues $\sqrt { - 1} $ and $-\sqrt { - 1} $, respectively. It is straightforward to verify that the vectors
$${Z_A} = \frac{1}{{\sqrt 2 }}\left( {{e_a} - \sqrt { - 1} J{e_a}} \right),\overline {{Z_A}}  = \frac{1}{{\sqrt 2 }}\left( {{e_a} + \sqrt { - 1} J{e_a}} \right),a=1,\cdots,n$$
form bases of ${V^{1,0}} $ and $ {V^{0,1}}$, respectively.

For a real tensor $T$, we do not distinguish between $T$ and its complexification $T^{\mathbb{C}}$. For example, the inner product $g$, extended complex-bilinearly on ${V^{\mathbb{C}}}$ and also denoted by $g$,  is symmetric, and we have
\begin{align*}
{g}\left( {{Z_A},{Z_B}} \right) =& {g}\left( {\frac{1}{{\sqrt 2 }}\left( {{e_a} - \sqrt { - 1} J{e_a}} \right),\frac{1}{{\sqrt 2 }}\left( {{e_b} - \sqrt { - 1} J{e_b}} \right)} \right)\\
 =& \frac{1}{2}\left( {g\left( {{e_a},{e_b}} \right) - g\left( {J{e_a},J{e_b}} \right) - \sqrt { - 1} g\left( {J{e_a},{e_b}} \right) - \sqrt { - 1} g\left( {{e_a},J{e_b}} \right)} \right)\\
 =& 0
 \end{align*}
and
\begin{align*}
{g}\left( {{Z_A},\overline {{Z_B}} } \right) =& {g}\left( {\frac{1}{{\sqrt 2 }}\left( {{e_a} + \sqrt { - 1} J{e_a}} \right),\frac{1}{{\sqrt 2 }}\left( {{e_b} - \sqrt { - 1} J{e_b}} \right)} \right)\\
 =& \frac{1}{2}\left( {g\left( {{e_a},{e_b}} \right) + g\left( {J{e_a},J{e_b}} \right) - \sqrt { - 1} g\left( {J{e_a},{e_b}} \right) - \sqrt { - 1} g\left( {{e_a},J{e_b}} \right)} \right)\\
 =& {\delta _{AB}}.
\end{align*}
Moreover, the norm of a real tensor $T$ coincides with the norm of its complexification:
$$\left| T \right|_\mathbb{C}^2 = g\left( {T,\overline T } \right) = \sum\limits_{{i_1} \cdots {i_k}}^{} {{T_{{i_1} \cdots {i_k}}}\overline {{T_{{i_1} \cdots {i_k}}}} }  = \sum\limits_{{i_1} \cdots {i_k}}^{} {{T_{{i_1} \cdots {i_k}}}{T_{{i_1} \cdots {i_k}}}}.$$
In this paper, we stipulate the following conventions for component indices of complex tensors  $T$:
$${T_{ijkl's't' \cdots }} = T\left( {{e_i},{e_j},{e_k},J{e_l},J{e_s},J{e_t}, \cdots } \right),1 \le i,j,k,l,s,t \le 2n, $$
$${T_{ABC\overline D \overline E \overline F  \cdots }} = T\left( {{Z_A},{Z_B},{Z_C},\overline {{Z_D}} ,\overline {{Z_E}} ,\overline {{Z_F}} , \cdots } \right),1 \le A,B,C,D,E,F \le n, $$
$${T_{abcd'e'f' \cdots }} = T\left( {{e_a},{e_b},{e_c},J{e_d},J{e_e},J{e_f}, \cdots } \right),1 \le a,b,c,d,e,f \le n. $$
The index sets  $\left\{ a,b,c,d,e,f,\cdots \right\}$ and $\left\{ A,B,C,D,E,F,\cdots\right\}$ are related via the correspondence ${Z_A} = \frac{1}{{\sqrt 2 }}\left( {{e_a} - \sqrt { - 1} J{e_a}} \right)$.

\begin{remark}
\label{remark2.1}
Let $\varphi$ be a tensor in $\otimes^2 {V^ {\mathbb{C}}}$. Then
$$\sum\limits_{i = 1}^{2n} {\varphi \left( {{e_i},{e_i}} \right) = \sum\limits_{A = 1}^n {\left( {\varphi \left( {{Z_A},\overline {{Z_A}} } \right) + \varphi \left( {\overline {{Z_A}} ,{Z_A}} \right)} \right)} }. $$
\end{remark}

\begin{remark}
\label{remark2.2}
Let $T$  be a real tensor on $\left( V,g,J \right)$, and let $\left\{ {{e_i}} \right\}_{i = 1}^{2n}$ be an orthonormal basis of  $V$ satisfying \eqref{equation2.1}. Then for ${Z_A} \odot {Z_B} \in { \odot ^2}{V^{1,0}}$, one has
\begin{align*}
\overline {{{\left( {\left( {{Z_A} \odot {Z_B}} \right)T} \right)}_{{i_1} \cdots {i_k}}}}  =& \sum\limits_{j = 1}^k {\sum\limits_{t = 1}^{2n} {\overline {{{\left( {{Z_A} \odot {Z_B}} \right)}_{{i_j}t}}{T_{{i_1} \cdots t \cdots {i_k}}}} } } \\
 =& \sum\limits_{j = 1}^k {\sum\limits_{t = 1}^{2n} {\overline {{{\left( {{Z_A} \odot {Z_B}} \right)}_{{i_j}t}}} {T_{{i_1} \cdots t \cdots {i_k}}}} } \\
 =& \sum\limits_{j = 1}^k {\sum\limits_{t = 1}^{2n} {{{\left( {\overline {{Z_A}}  \odot \overline {{Z_B}} } \right)}_{{i_j}t}}{T_{{i_1} \cdots t \cdots {i_k}}}} } \\
 =& {\left( {\left( {\overline {{Z_A}}  \odot \overline {{Z_B}} } \right)T} \right)_{{i_1} \cdots {i_k}}}.
\end{align*}
\end{remark}

\begin{lemma}
\label{lemma2.3}
Let $S \in {\odot} ^2 V^{1,0}$, then
$$S_{i'j'} =  - S_{ij} ,S_{ij'} = S_{i'j} ,$$
that is to say if we see $S \in End\left( V^ {\mathbb{C}} \right)$, then
$$S \circ J + J \circ S = 0.$$
\end{lemma}
\begin{proof}
Take an orthonormal basis $\{Z_A\}$ of $V^{1,0}$.  Write $S$ as
$$S = \sum\limits_{A,B = 1}^n {{S_{AB}}{Z_A} \otimes {Z_B}}  = \frac{1}{2}\sum\limits_{a,b = 1}^n {{S_{AB}}\left( {{e_a} \otimes {e_b} - J{e_a} \otimes J{e_b} - \sqrt { - 1} {e_a} \otimes J{e_b} - \sqrt { - 1} J{e_a} \otimes {e_b}} \right)} $$
Thus
$$S_{ab}  =  - S_{a'b'} = \frac12S_{AB},\ \ \ S_{ab'}  = S_{a'b} = - \frac12\sqrt { - 1} S_{AB}.$$
\end{proof}

\subsection{Algebraic curvature tensor}
Let $T \in { \otimes ^4}{V^ * }$. We say that $T$ is an algebraic curvature tensor on $V$ if it satisfies the following symmetries:
$${T_{ijkl}} =  - {T_{jikl}} =  - {T_{ijlk}} = {T_{klij}},$$
and the first Bianchi identity:
$${T_{ijkl}} + {T_{iklj}} + {T_{iljk}} = 0.$$
If, in addition, $T$ satisfies
$$T\left( {J{e_i},J{e_j},{e_k},{e_l}} \right) = T\left( {{e_i},{e_j},{e_k},{e_l}} \right),$$
then we call $T$  a K{\"a}hler algebraic curvature tensor. Considering the complexification of $T$, it is straightforward to see that
$$T\left( {{Z_A},{Z_B}, \cdot , \cdot } \right) = T\left( { \cdot , \cdot ,\overline {{Z_C}} ,\overline {{Z_D}} } \right) = 0$$
and
$$T\left( {{Z_A},\overline {{Z_C}} ,\overline {{Z_D}} ,{Z_B}} \right) = T\left( {{Z_A},\overline {{Z_D}} ,\overline {{Z_C}} ,{Z_B}} \right).$$
\begin{remark}
\label{remark2.4}
In K{\"a}hler geometry, the K{\"a}hler algebraic curvature tensor is originally defined on ${V^{1,0}} \times {V^{0,1}} \times {V^{0,1}} \times {V^{1,0}}$. Here, however, we may consider it as the result of extending it to the entire product  ${\left({V^\mathbb{C}}\right)}^4$.
\end{remark}

The Ricci tensor of $T$, denoted by  $Ric^T$, is defined as the $\left( 2,4 \right)$ contraction of $T$. That is,
$${R^T_{ik}} = Ric^T\left( {{e_i},{e_k}} \right) = \sum\limits_{j = 1}^{2n} {T\left( {{e_i},{e_j},{e_k},{e_j}} \right)}  = \sum\limits_{j = 1}^{2n} {{T_{ijkj}}}.$$
If $T$ is a K${\ddot a}$hler algebraic curvature tensor, then
\begin{equation}
\begin{aligned}
Ric^T\left( {{e_i},{e_k}} \right) =& \sum\limits_{j = 1}^{2n} {T\left( {{e_i},{e_j},{e_k},{e_j}} \right)} \\
 =& \sum\limits_{j = 1}^n {\left( {T\left( {{e_i},{e_j},{e_k},{e_j}} \right) + T\left( {{e_i},J{e_j},{e_k},J{e_j}} \right)} \right)} \\
 =& \sum\limits_{j = 1}^n {\left( {T\left( {{e_i},{e_j},J{e_k},J{e_j}} \right) - T\left( {{e_i},J{e_j},J{e_k},{e_j}} \right)} \right)} \\
 =& \sum\limits_{j = 1}^n {T\left( {{e_i},J{e_k},{e_j},J{e_j}} \right)} .
\end{aligned}
\end{equation}

Note that every K{\"a}hler algebraic curvature tensor $T$ defines a self-adjoint operator
 $${\mathcal{C}^T}\colon { \odot ^2}{V^{1,0}}\rightarrow { \odot ^2}{V^{1,0}}$$
 via
$${\mathcal{C}^T}\left( {{Z_A} \odot {Z_B}} \right) = \sum\limits_{C,D = 1}^n {{T_{A\overline C \overline D B}}{Z_C} \odot {Z_D}},$$
and similarly defines a self-adjoint operator $\overline{\mathcal{C}^T}$ on ${ \odot ^2}{V^{0,1}}$,
$$\overline {\mathcal{C}^T} \left( {\overline {{Z_A}}  \odot \overline {{Z_B}} } \right) = \sum\limits_{C,D = 1}^n {{T_{\overline A CD\overline B }}\overline {{Z_C}}  \odot \overline {{Z_D}} }. $$
We may regard the operator $\overline {\mathcal{C}^T}$ as the complex conjugate of ${\mathcal{C}^T}$, i.e.,
$$\overline {\mathcal{C}^T} \left( {\overline {{Z_A}}  \odot \overline {{Z_B}} } \right): = \overline {{\mathcal{C}^T}\left( {{Z_A} \odot {Z_B}} \right)},$$
 therefore, $\overline {\mathcal{C}^T}$ and ${\mathcal{C}^T}$ have the same eigenvalues, and their eigenvectors are conjugate to each other.

Let $R$ be a K{\"a}hler algebraic curvature tensor on $\left( M^n,g,J \right)$. Denote the corresponding self-adjoint operator ${\mathcal{C}^R}$ ( simply  as $\mathcal{C}$)  defined as above. For any orthonormal basis $\left\{ {{S^\alpha }} \right\}_{\alpha  = 1}^N$ of ${ \odot ^2}{V^{1,0}}$, $N = \frac{{n(n + 1)}}{2}$, and for any tensor $T$, we define
$${T^{{ \odot ^2}{V^{1,0}}}} = \sum\limits_{\alpha  = 1}^N {{S^\alpha }\left( T \right) \otimes \overline{S^\alpha }} ,$$
and
$$\mathcal{C}\left( {{T^{{ \odot ^2}{V^{1,0}}}}} \right) = \sum\limits_{\alpha  = 1}^N {{S^\alpha }\left( T \right) \otimes \overline {\mathcal{C}} \left( {\overline {{S^\alpha }} } \right)}.$$
Consequently, if $\left\{ {{S^\alpha }} \right\}_{\alpha  = 1}^N$ is an orthonormal eigenbasis of $\mathcal{C}$ with corresponding eigenvalues $\left\{ {{\lambda _\alpha }} \right\}_{\alpha  = 1}^N$, then
$$g\left( {\mathcal{C}\left( {{T^{{ \odot ^2}{V^{1,0}}}}} \right),\overline {{T^{{ \odot ^2}{V^{1,0}}}}} } \right) = \sum\limits_{\alpha  = 1}^N {{\lambda _\alpha }{{\left| {{S^\alpha }T} \right|}^2}}. $$

\subsection{The cone condition and an eigenvalue inequality}
The notion of a cone condition for curvature operators was first introduced by Li\cite{li2025new}. For a curvature operator $\mathcal{R}$ on a vector space $V$, let $${\lambda _1} \le {\lambda _2} \le  \cdots  \le {\lambda _N}$$ be the eigenvalues of $\mathcal{R}$. We say that $\mathcal{R}$ belongs to the cone $C\left( \alpha ,\theta \right)$ if
$${\lambda _1} +  \cdots  + {\lambda _\alpha } \ge  - \alpha \theta \bar \lambda ,$$
where $1 \le \alpha  \le N$ and $\theta  \ge  - 1$. Analogously, we write  $\mathcal{R} \in \mathring{C}\left( \alpha ,\theta \right)$ when the inequality is strict, and $\mathcal{R} \in \partial C\left( \alpha ,\theta \right)$ when the equality holds.

The following inequality is taken from \cite{fu2025new}. For a real number $\alpha $, write ${\alpha _0} = \alpha  - \left[ \alpha  \right]$ and adopt the conventions:
$$\sum\limits_{\beta = 1}^{\alpha}  {{\lambda _\beta }}  = {\lambda _1} +  \cdots  + {\lambda _{\left[ \alpha  \right]}} + {\alpha _0}{\lambda _{\left[ \alpha  \right] + 1}}$$
and
$$\sum\limits_{\beta  = \alpha  + 1}^N {{\lambda _\beta } = \left( {1 - {\alpha _0}} \right)} {\lambda _{\left[ \alpha  \right] + 1}} + {\lambda _{\left[ \alpha  \right] + 2}} \cdots  + {\lambda _N}.$$
\begin{lemma} $(${\cite{fu2025new}}$)$
\label{lemma2.5}
Let ${\lambda _1} \le {\lambda _2} \le  \cdots  \le {\lambda _N}$ be some real numbers with $\sum\limits_{\beta  = 1}^N {{\lambda _\beta }}  = N\bar \lambda $, i.e., $\bar \lambda $ is their average. If the smallest $\alpha$  numbers satisfy
$${\lambda _1} +  \cdots  + {\lambda _\alpha } \ge  - \alpha \theta \bar \lambda ,$$
then for any nonnegative real numbers ${M_1},{M_2}, \cdots ,{M_N}$, denote $S = \sum\limits_{\beta  = 1}^N {{M_\beta }} $ and $M = \max \left\{ {{M_1},{M_2}, \cdots ,{M_N}} \right\}$. Provided that $\alpha  \le \frac{S}{M}$, we have
$$\sum\limits_{\beta = 1}^N {{\lambda _\beta }{M_\beta}}  \ge  - S\theta \bar \lambda .$$
\end{lemma}
\begin{proof}
First,
\begin{align*}
\sum\limits_{\beta  = 1}^N {{\lambda _\beta}{M_\beta }}  =& \sum\limits_{\beta = 1}^{\alpha} {{\lambda _\beta }{M_\beta}}  + \sum\limits_{\beta  = \alpha + 1}^N {{\lambda _\beta }{M_\beta }} \\
 \ge & \sum\limits_{\beta = 1}^{\alpha}{{\lambda _\beta }{M_\beta}}  + {\lambda _{\left[ \alpha \right] + 1}}\sum\limits_{\beta  = \alpha + 1}^N {{M_\beta }} \\
 = & \sum\limits_{\beta  = 1}^{\alpha} {\left( {{\lambda _\beta } - {\lambda _{\left[ \alpha \right] + 1}}} \right){M_\beta }}  + {\lambda _{\left[ \alpha \right] + 1}}\sum\limits_{\beta  = 1}^N {{M_\beta }} \\
 \ge & M\sum\limits_{\beta  = 1}^{\alpha} {\left( {{\lambda _\beta } - {\lambda _{\left[ \alpha \right] + 1}}} \right)}  + {\lambda _{\left[ \alpha \right] + 1}}S\\
= & {\lambda _{\left[ \alpha\right] + 1}}\left( {S - \alpha M} \right) + M\sum\limits_{\beta  =   1}^{\alpha} {{\lambda _\beta }} .
\end{align*}
Second, using the ordering ${\lambda _1} \le {\lambda _2} \le  \cdots  \le {\lambda _N}$ together with the hypothesis ${\lambda _1} +  \cdots  + {\lambda _\alpha} \ge  - \alpha\theta \bar \lambda$, we obtain
\begin{equation}
\label{equation2.3}
{\lambda _{\left[ \alpha \right] + 1}} \ge \frac{{{\lambda _1} +  \cdots  + {\lambda _\alpha}}}{\alpha} \ge  - \theta \bar \lambda .
\end{equation}
Consequently, whenever $\alpha \le \frac{S}{M}$,
$$\sum\limits_{\beta  = 1}^N {{\lambda _\beta }{M_\beta }}  \ge  - \theta \bar \lambda \left( {S - \alpha M} \right) - \alpha M\theta \bar \lambda  \ge  - S\theta \bar \lambda. $$
\end{proof}

\section{proof of Theorem~\ref{theorem1.1}}
\begin{proof}[{\bf Proof of Theorem~\ref{theorem1.1}}]
(1) Take the orthonormal basis of ${ \odot ^2}{V^{1,0}}$ given by $${\left\{ {\frac{1}{{\sqrt 2 }}\left( {{Z_A} \odot {Z_B}} \right)} \right\}_{1 \le A < B \le n}} \cup \left\{ {\frac{1}{2}\left( {{Z_A} \odot {Z_A}} \right)} \right\}_{A = 1}^n,$$ then

\begin{align*}
  &\sum\limits_{1 \le A < B \le n}^{} {g\left( {\mathcal{C} \left( {\frac{1}{{\sqrt 2 }}\left( {{Z_A} \odot {Z_B}} \right)} \right),\frac{1}{{\sqrt 2 }}\left( {\overline {{Z_A}}  \odot \overline {{Z_B}} } \right)} \right)}  + \sum\limits_{A = 1}^n {g\left( {\mathcal{C}\left( {\frac{1}{2}\left( {{Z_A} \odot {Z_A}} \right)} \right),\frac{1}{2}\left( {\overline {{Z_A}}  \odot \overline {{Z_A}} } \right)} \right)}  \\
  &= \frac{1}{4}\sum\limits_{A,B = 1}^n {g\left( {\mathcal{C}\left( {{Z_A} \odot {Z_B}} \right),\overline {{Z_A}}  \odot \overline {{Z_B}} } \right)}  =\sum\limits_{A,B = 1}^n {R\left( {{Z_A},\overline {{Z_A}} ,\overline {{Z_B}} ,{Z_B}} \right)}\\
   &=\frac{1}{2}\sum\limits_{a,B = 1}^n {R\left( {{e_a} - \sqrt { - 1} J{e_a},{e_a} + \sqrt { - 1} J{e_a},\overline {{Z_B}} ,{Z_B}} \right)}
   = \sum\limits_{a,B = 1}^n {R\left( {{e_a},\sqrt { - 1} J{e_a},\overline {{Z_B}} ,{Z_B}} \right)} \\
   &= \sum\limits_{a,b = 1}^n {R\left( {{e_a},J{e_a},{e_b},J{e_b}} \right)}
   = \sum\limits_{a = 1}^n {Ric\left( {{e_a},{e_a}} \right)}
   =\frac{s}{2} ,
\end{align*}
because $M$ has $\frac{{n\left( {n + 1} \right)}}{2}$-positive ( nonnegative) Calabi curvature operator,  the scalar curvature is positive(nonnegative).

(2) Let $A \in \left\{ {1,2, \cdots ,n} \right\}$ be fixed. For all $B \ne A$, define
$$2{\mu _b}=g\left( {\mathcal{C}\left( {\frac{1}{{\sqrt 2 }}\left( {{Z_A} \odot {Z_B}} \right)} \right),\frac{1}{{\sqrt 2 }}\left( {\overline {{Z_A}}  \odot \overline {{Z_B}} } \right)} \right) = 2R\left( {{e_a},J{e_a},{e_b},J{e_b}} \right),$$
$${\mu _a}=g\left( {\mathcal{C}\left( {\frac{1}{2}\left( {{Z_A} \odot {Z_A}} \right)} \right),\frac{1}{2}\left( {\overline {{Z_A}}  \odot \overline {{Z_A}} } \right)} \right) = R\left( {{e_a},J{e_a},{e_a},J{e_a}} \right).$$
Arranging these $n$ numbers in ascending order yields
$$2{\mu _1}, \cdots 2{\mu _{k - 1}},{\mu _k} = {\mu _a},2{\mu _{k + 1}}, \cdots, 2{\mu _n}.$$
Because $M$ has $\frac{n+1}{2}$- positive(nonnegative) Calabi curvature operator, we obtain the following estimates.

Case 1. $k\le \frac{n+1}{2}$.  When $n$ is odd,
\begin{align*}
  0 \le & \sum\limits_{i = 1}^{k - 1} {2\mu _i}  + {\mu _k} + \sum\limits_{j = k + 1}^{\frac{{n + 1}}{2}} {2{\mu _j}} \\
   \le & \sum\limits_{i = 1}^{k - 1} {\left( {{\mu _i} + {\mu _{n - i + 1}}} \right)}  + {\mu _k} + \sum\limits_{j = k + 1}^{\frac{{n + 1}}{2}} {\left( {{\mu _j} + {\mu _{n - j + 2}}} \right)} \\
   =& \sum\limits_{l = 1}^n {{\mu _l}}  = \sum\limits_{b = 1}^n {R\left( {{e_a},J{e_a},{e_b},J{e_b}} \right)}
   =Ric\left( {{e_a},{e_a}} \right).
\end{align*}
When $n$ is even,
\begin{align*}
  0 \le & \sum\limits_{i = 1}^{k - 1} {2\mu _i}  + {\mu _k} + \sum\limits_{j = k + 1}^{\frac{n}{2}} {2{\mu _j}}  + {\mu _{\frac{n}{2} + 1}}  \\
   \le & \sum\limits_{i = 1}^{k - 1} {\left( {{\mu _i} + {\mu _{n - i + 1}}} \right)}  + {\mu _k} + \sum\limits_{j = k + 1}^{\frac{n}{2}} {\left( {{\mu _j} + {\mu _{n - j + 2}}} \right)}  + {\mu _{\frac{n}{2} + 1}} \\
   =& \sum\limits_{l = 1}^n {{\mu _l}}  = \sum\limits_{b = 1}^n {R\left( {{e_a},J{e_a},{e_b},J{e_b}} \right)}
   = Ric\left( {{e_a},{e_a}} \right) .
\end{align*}
Case 2.  $k > \frac{n+1}{2}$. When $n$ is odd,
\begin{align*}
0 \le & \sum\limits_{i = 1}^{\frac{{n + 1}}{2}} {2{\mu _i}}
 \le  \sum\limits_{i = 1}^{n - k} {\left( {{\mu _i} + {\mu _{n - i + 1}}} \right)}  + \left( {{\mu _{n - k + 1}} + {\mu _{n - k + 2}} + {\mu _k}} \right) + \sum\limits_{j = n - k + 3}^{\frac{{n + 1}}{2}} {\left( {{\mu _j} + {\mu _{n - j + 2}}} \right)} \\
 = & \sum\limits_{l = 1}^n {{\mu _l}}  = \sum\limits_{b = 1}^n {R\left( {{e_a},J{e_a},{e_b},J{e_b}} \right)}
 =  Ric\left( {{e_a},{e_a}} \right).
\end{align*}
When $n$ is even and $k \ne \frac{n}{2}+1$,
\begin{align*}
0 \le & \sum\limits_{i = 1}^{\frac{n}{2}} {2{\mu _i}}  + {\mu _{\frac{n}{2} + 1}}\\
 \le & \sum\limits_{i = 1}^{n - k} {\left( {{\mu _i} + {\mu _{n - i + 1}}} \right)}  + \left( {{\mu _{n - k + 1}} + {\mu _{n - k + 2}} + {\mu _k}} \right) + \sum\limits_{j = n - k + 3}^{\frac{n}{2}} {\left( {{\mu _j} + {\mu _{n - j + 2}}} \right)}  + {\mu _{\frac{n}{2} + 1}}\\
 = & \sum\limits_{l = 1}^n {{\mu _l}}  = \sum\limits_{b = 1}^n {R\left( {{e_a},J{e_a},{e_b},J{e_b}} \right)}
 =  Ric\left( {{e_a},{e_a}} \right).
\end{align*}
When $n$ is even and $k = \frac{n}{2}+1$,
\begin{align*}
0 \le & \sum\limits_{i = 1}^{\frac{n}{2}} {2{\mu _i}}  + \frac{1}{2}{\mu _{\frac{n}{2} + 1}}\\
 \le & \sum\limits_{i = 1}^{\frac{n}{2} - 1} {\left( {{\mu _i} + {\mu _{n - i + 1}}} \right)}  + \left( {{\mu _{\frac{n}{2}}} + \frac{1}{2}{\mu _{\frac{n}{2} + 1}}} \right) + \frac{1}{2}{\mu _{\frac{n}{2} + 1}}\\
 = & \sum\limits_{l = 1}^n {{\mu _l}}  = \sum\limits_{b = 1}^n {R\left( {{e_a},J{e_a},{e_b},J{e_b}} \right)}
 =  Ric\left( {{e_a},{e_a}} \right).
\end{align*}
Combining all cases and using   $Ric\left( {J{e_a},J{e_a}} \right) = Ric\left( {{e_a},{e_a}} \right)$, we can conclude that $Ric \ge 0$.

(3) Since $M$ has $\left( n-1 \right)$-positive(nonnegative) Calabi curvature operator, we obtain
\begin{align*}
0 \le & \sum\limits_{b \ne a}^{} {2{\mu _b}}
 = \sum\limits_{b = 1}^n {2{\mu _b}}  - 2{\mu _a}\\
 =& 2\sum\limits_{b = 1}^n {R\left( {{e_a},J{e_a},{e_b},J{e_b}} \right)}  - 2R\left( {{e_a},J{e_a},{e_a},J{e_a}} \right)\\
 =& 2\left( {Ric\left( {{e_a},{e_a}} \right) - R\left( {{e_a},J{e_a},{e_a},J{e_a}} \right)} \right)
 = 2Ri{c^ \bot }\left( {{e_a},{e_a}} \right) .
\end{align*}

(4) Since $M$ has $ n $-positive(nonnegative) Calabi curvature operator, we obtain
\begin{align*}
0 \le & \sum\limits_{b \ne a}^{} {2{\mu _b}}  + {\mu _a}
 =  \sum\limits_{b = 1}^n {2{\mu _b}}  - {\mu _a}\\
 = & 2\sum\limits_{b = 1}^n {R\left( {{e_a},J{e_a},{e_b},J{e_b}} \right)}  - R\left( {{e_a},J{e_a},{e_a},J{e_a}} \right)\\
 = & 2Ric\left( {{e_a},{e_a}} \right) - R\left( {{e_a},J{e_a},{e_a},J{e_a}} \right).
\end{align*}

(5) Suppose $M$ is reducible, and write $M = M_1 \times M_2$, where $M_i$ is a complex $n_i$-dimensional K{\"a}hler manifold, $i=1,2$. Choose an orthonormal basis $$\left\{ {{e_1}, \cdots ,{e_{{n_1}}}} , {{Je_1}, \cdots ,{Je_{{n_1}}}} \right\}$$ for $M_1$ and $$\left\{ {{e_{{n_1} + 1}}, \cdots ,{e_{{n_1} + {n_2}}}},{{Je_{{n_1} + 1}}, \cdots ,{Je_{{n_1} + {n_2}}}} \right\}$$ for $ M_2$. For indices $1 \le A \le {n_1}$ and ${n_1} + 1 \le C \le {n_1} + {n_2}$, the curvature tensor of a product satisfies
$$ R\left( {{Z_A},\overline {{Z_A}} ,\overline {{Z_C}} ,{Z_C}} \right) = 0, $$
hence
$$g\left( {\mathcal{C}\left( {\frac{1}{{\sqrt 2 }}{Z_A} \odot {Z_C}} \right),\frac{1}{{\sqrt 2 }}\overline {{Z_A}}  \odot \overline {{Z_C}} } \right) =0.$$
Consequently,
$$\sum\limits_{\alpha  = 1}^{{n_1}  {n_2}} {{\lambda _\alpha }}  \le \sum\limits_{A = 1}^{{n_1}} {\sum\limits_{C = {n_1} + 1}^{{n_1} + {n_2}} {g\left( {\mathcal{C}\left( {\frac{1}{{\sqrt 2 }}{Z_A} \odot {Z_C}} \right),\frac{1}{{\sqrt 2 }}\overline {{Z_A}}  \odot \overline {{Z_C}} } \right)} }  = 0.$$
Thus the Calabi operator  $\mathcal{C}$ is $ {n_1}{n_2} $-nonpositive . This contradicts the hypothesis that $${\lambda _1} + {\lambda _2} +  \cdots  + {\lambda _\alpha} > 0\ \ \  \text{for}\ \   \alpha \le n-1.$$ Therefore $M$ must be irreducible.
\end{proof}


\section{Bochner formula for K{\"a}hler algebraic curvature tensors}
In this section, we adapt the method of \cite{dai2024einstein} to establish a Bochner-type formula for the Calabi operator $\mathcal{C}$ associated with an arbitrary K{\"a}hler algebraic curvature tensor on a given complex $n$-dimensional K{\"a}hler manifold $\left( {M,g,J} \right)$.

In \cite{2021New}, Petersen and Wink gave a new method to study curvature operators. Using this method, we can arrive at the following proposition.

\begin{proposition}
\label{pro4.1}
Let $T$ be a K{\"a}hler algebraic curvature tensor on a  K{\"a}hler manifold $\left( {M,g,J} \right)$ of complex dimension $n$. Then
\begin{equation*}
\begin{aligned}
g\left( {{\mathcal{C} }\left( {{T^{{ \odot ^2}{V^{1,0}}}}} \right),\overline {{T^{{ \odot ^2}{V^{1,0}}}}} } \right)  =& 4\sum\limits_{A,B,C,D=1}^{n} {\sum\limits_{i,j=1}^{2n} {{R_{\overline A C  D \overline B}} {{T_{A\overline C ij}}{T_{B\overline D ij}}}} }\\
 &+ 4\sum\limits_{B,C,D=1}^{n} {\sum\limits_{j,k,l=1}^{2n} { {  {R_{\overline D C  D \overline B}}{T_{Bijk}}{T_{\overline C ijk}}} } }\\
 =&   - \sum\limits_{i,j,s,t,p,q = 1}^{2n} {{R_{sptq}}{T_{spij}}{T_{tqij}}}  + 2\sum\limits_{i,j,k,s,t = 1}^{2n} {{R_{ts}}{T_{tijk}}{T_{sijk}}} .
\end{aligned}
\end{equation*}
\end{proposition}
\begin{proof}
Take the orthonormal basis of ${{ \odot ^2}{V^{1,0}}}$
$$\frac{1}{{\sqrt 2 }}{\left\{ {{Z_A} \odot {Z_B}} \right\}_{A < B}} \cup \frac{1}{2}\left\{ {{Z_A} \odot {Z_A}} \right\}_{A=1}^n.$$
 Using the Remark \ref{remark2.1} and \ref{remark2.2}, we compute

\begin{align*}
&g\left( {\left( {{Z_A} \odot {Z_B}} \right)T,\overline {\left( {{Z_C} \odot {Z_D}} \right)T} } \right)
 = g\left( {\left( {{Z_A} \odot {Z_B}} \right)T,\left( {\overline {{Z_C}}  \odot \overline {{Z_D}} } \right)T} \right)\\
   =& \sum\limits_{i,j,k,l = 1}^{2n} {\left( {T\left( {{Z_A} \odot {Z_B}\left( {{e_i}} \right),{e_j},{e_k},{e_l}} \right) +  \cdots  + T\left( {{e_i},{e_j},{e_k},{Z_A} \odot {Z_B}\left( {{e_l}} \right)} \right)} \right)}  \\
  &\left( {T\left( {\overline {{Z_C}}  \odot \overline {{Z_D}} \left( {{e_i}} \right),{e_j},{e_k},{e_l}} \right) +  \cdots  + T\left( {{e_i},{e_j},{e_k},\overline {{Z_C}}  \odot \overline {{Z_D}} \left( {{e_l}} \right)} \right)} \right) \hfill \\
   =& 4\sum\limits_{E = 1}^n {\sum\limits_{j,k,l = 1}^{2n} {T\left( {{Z_A} \odot {Z_B}\left( {\overline {{Z_E}} } \right),{e_j},{e_k},{e_l}} \right)} }   \\
  &\left( {T\left( {\overline {{Z_C}}  \odot \overline {{Z_D}} \left( {{Z_E}} \right),{e_j},{e_k},{e_l}} \right) + T\left( {{Z_E},\overline {{Z_C}}  \odot \overline {{Z_D}} \left( {{e_j}} \right),{e_k},{e_l}} \right) + 2T\left( {{Z_E},{e_j},\overline {{Z_C}}  \odot \overline {{Z_D}} \left( {{e_k}} \right),{e_l}} \right)} \right) \\
 =& 4\sum\limits_{E,F=1}^n {\sum\limits_{i,j = 1}^{2n} {\left( {T\left( {{Z_A} \odot {Z_B}\left( {\overline {{Z_E}} } \right),\overline {{Z_F}} ,{e_i},{e_j}} \right)T\left( {{Z_E},\overline {{Z_C}}  \odot \overline {{Z_D}} \left( {{Z_F}} \right),{e_i},{e_j}} \right)} \right.} } \\
& \left. { + 2T\left( {{Z_A} \odot {Z_B}\left( {\overline {{Z_E}} } \right),{e_i},\overline {{Z_F}} ,{e_j}} \right)T\left( {{Z_E},{e_i},\overline {{Z_C}}  \odot \overline {{Z_D}} \left( {{Z_F}} \right),{e_j}} \right)} \right)\\
 &+ 4\sum\limits_{E=1}^n {\sum\limits_{i,j,k = 1}^{2n} {T\left( {{Z_A} \odot {Z_B}\left( {\overline {{Z_E}} } \right),{e_i},{e_j},{e_k}} \right)T\left( {\overline {{Z_C}}  \odot \overline {{Z_D}} \left( {{Z_E}} \right),{e_i},{e_j},{e_k}} \right)} } \\
 =& 4\sum\limits_{E,F=1}^n {\sum\limits_{i,j = 1}^{2n} {\left( {\left( {{\delta _{AE}}{T_{B\overline F ij}} + {\delta _{BE}}{T_{A\overline F ij}}} \right)\left( {{\delta _{CF}}{T_{E\overline D ij}} + {\delta _{DF}}{T_{E\overline C ij}}} \right)} \right.} } \\
&\left. { + 2\left( {{\delta _{AE}}{T_{Bi\overline F j}} + {\delta _{BE}}{T_{Ai\overline F j}}} \right)\left( {{\delta _{CF}}{T_{Ei\overline D j}} + {\delta _{DF}}{T_{Ei\overline C j}}} \right)} \right)\\
 &+ 4\sum\limits_{E=1}^n {\sum\limits_{i,j,k = 1}^{2n} {\left( {{\delta _{AE}}{T_{Bijk}} + {\delta _{BE}}{T_{Aijk}}} \right)} } \left( {{\delta _{CE}}{T_{\overline D ijk}} + {\delta _{DE}}{T_{\overline C ijk}}} \right)\\
 =& 8\sum\limits_{i,j = 1}^{2n} \left( {\left( {{T_{B\overline C ij}}{T_{A\overline D ij}} + {T_{B\overline D ij}}{T_{A\overline C ij}}} \right) + 2\left( {{T_{Bi\overline C j}}{T_{Ai\overline D j}} + {T_{Bi\overline D j}}{T_{Ai\overline C j}}} \right)} \right) \\
 &+ 4\sum\limits_{i,j,k = 1}^{2n} {\left( {{\delta _{AC}}{T_{Bijk}}{T_{\overline D ijk}} + {\delta _{AD}}{T_{Bijk}}{T_{\overline C ijk}} + {\delta _{BC}}{T_{Aijk}}{T_{\overline D ijk}} + {\delta _{BD}}{T_{Aijk}}{T_{\overline C ijk}}} \right)}.
\end{align*}
Because $T$ is a K{\"a}hler algebraic curvature tensor,
\begin{align*}
\sum\limits_{i,j = 1}^{2n} {{T_{Bi\overline C j}}{T_{Ai\overline D j}}}  =& \sum\limits_{i,j = 1}^{2n} {T\left( {{Z_B},{e_i},\overline {{Z_C}} ,{e_j}} \right)T\left( {{Z_A},{e_i},\overline {{Z_D}} ,{e_j}} \right)} \\
=& \sum\limits_{i,j = 1}^{2n} {T\left( {{Z_B},{e_i},\overline {{Z_C}} ,J{e_j}} \right)T\left( {{Z_A},{e_i},\overline {{Z_D}} ,J{e_j}} \right)}  \hfill \\
=& \sum\limits_{i,j = 1}^{2n} {T\left( {{Z_B},{e_i},J\overline {{Z_C}} ,{J^2}{e_j}} \right)T\left( {{Z_A},{e_i},J\overline {{Z_D}} ,{J^2}{e_j}} \right)}  \hfill \\
=&  - \sum\limits_{i,j = 1}^{2n} {T\left( {{Z_B},{e_i},\overline {{Z_C}} ,{e_j}} \right)T\left( {{Z_A},{e_i},\overline {{Z_D}} ,{e_j}} \right)}  \hfill \\
=& 0.
\end{align*}
Hence
\begin{align*}
 &g\left( {\left( {{Z_A} \odot {Z_B}} \right)T,\overline {\left( {{Z_C} \odot {Z_D}} \right)T} } \right)
 = 8\sum\limits_{i,j = 1}^{2n} \left( {{T_{B\overline C ij}}{T_{A\overline D ij}} + {T_{B\overline D ij}}{T_{A\overline C ij}}} \right)   \\
 &+ 4\sum\limits_{i,j,k = 1}^{2n} {\left( {{\delta _{AC}}{T_{Bijk}}{T_{\overline D ijk}} + {\delta _{AD}}{T_{Bijk}}{T_{\overline C ijk}} + {\delta _{BC}}{T_{Aijk}}{T_{\overline D ijk}} + {\delta _{BD}}{T_{Aijk}}{T_{\overline C ijk}}} \right)} .
\end{align*}
Summing the above terms gives
\begin{align*}
&g\left( {{\mathcal{C} }\left( {{T^{{ \odot ^2}{V^{1,0}}}}} \right),\overline {{T^{{ \odot ^2}{V^{1,0}}}}} } \right)
= \frac{1}{16}\sum\limits_{A,B,C,D=1}^{n} {g\left( {\left( {{Z_A} \odot {Z_B}} \right)T,\left( {\overline {{Z_C}}  \odot \overline {{Z_D}} } \right)T} \right) {g\left( {\overline {\mathcal{C}} \left( {\overline {{Z_A}}  \odot \overline {{Z_B}} } \right),{Z_C} \odot {Z_D}} \right)}} \\
 =& 2\sum\limits_{A,B,C,D=1}^{n} {\sum\limits_{i,j=1}^{2n} {{R_{\overline A C  D \overline B}}\left(  {{T_{B\overline C ij}}{T_{A\overline D ij}} + {T_{B\overline D ij}}{T_{A\overline C ij}}  } \right)} }  \\
 &+\sum\limits_{A,B,C,D=1}^{n} {\sum\limits_{i,j,k=1}^{2n} {{R_{\overline A C  D \overline B}}\left( {{\delta _{AC}}{T_{Bijk}}{T_{\overline D ijk}} + {\delta _{AD}}{T_{Bijk}}{T_{\overline C ijk}} + {\delta _{BC}}{T_{Aijk}}{T_{\overline D ijk}} + {\delta _{BD}}{T_{Aijk}}{T_{\overline C ijk}}} \right)} } \\
 =& 4\sum\limits_{A,B,C,D=1}^{n} {\sum\limits_{i,j=1}^{2n} {{R_{\overline A C  D \overline B}} {{T_{A\overline C ij}}{T_{B\overline D ij}} } } } + 4\sum\limits_{B,C,D=1}^{n} {\sum\limits_{i,j,k=1}^{2n} { {  {R_{\overline D C  D \overline B}}{T_{Bijk}}{T_{\overline C ijk}}} } } .
\end{align*}
We now evaluate the two terms separately. For the first term, using Remark \ref{remark2.1},
\begin{align*}
 & \sum\limits_{A,B,C,D = 1}^n {\sum\limits_{i,j = 1}^{2n} {{R_{\bar ACD\bar B}}{T_{A\bar Cij}}{T_{B\bar Dij}}} }  = \sum\limits_{C,D = 1}^n {\sum\limits_{i,j,s,t = 1}^{2n} {{R_{sCDt}}{T_{s\bar Cij}}{T_{t\bar Dij}}} }\\
   =& \frac{1}{2} \sum\limits_{c,D = 1}^n {\sum\limits_{i,j,s,t = 1}^{2n} {\left( {{R_{scDt}} - \sqrt { - 1} {R_{sc'Dt}}} \right)\left( {{T_{scij}} + \sqrt { - 1} {T_{sc'ij}}} \right){T_{t\bar Dij}}} } \\
       =& \frac{1}{2} \sum\limits_{c,D = 1}^n {\sum\limits_{i,j,s,t = 1}^{2n} {\left( {{R_{scDt}}{T_{scij}} + {R_{sc'Dt}}{T_{sc'ij}} + \sqrt { - 1} {R_{scDt}}{T_{sc'ij}} - \sqrt { - 1} {R_{sc'Dt}}{T_{scij}}} \right){T_{t\bar Dij}}} } \\
    =& \frac{1}{2} \sum\limits_{D = 1}^n {\sum\limits_{i,j,s,t = 1}^{2n} {\left( {\sum\limits_{p = 1}^n {{R_{spDt}}\left( {{T_{spij}} + \sqrt { - 1} {T_{sp'ij}}} \right)}  + \sum\limits_{p = n + 1}^{2n} {{R_{spDt}}\left( {{T_{spij}} + \sqrt { - 1} {T_{sp'ij}}} \right)} } \right){T_{t\bar Dij}}} } \\
   =& \frac{1}{2} \sum\limits_{D = 1}^n {\sum\limits_{i,j,s,t,p = 1}^{2n} {{R_{spDt}}\left( {{T_{spij}} + \sqrt { - 1} {T_{sp'ij}}} \right){T_{t\bar Dij}}} }   \\
   =& \frac{1}{4} \sum\limits_{d = 1}^n {\sum\limits_{i,j,s,t,p = 1}^{2n} {\left( {{T_{spij}} + \sqrt { - 1} {T_{sp'ij}}} \right)\left( {{R_{spdt}} - \sqrt { - 1} {R_{spd't}}} \right)\left( {{T_{tdij}} + \sqrt { - 1} {T_{td'ij}}} \right)} }   \\
   =& \frac{1}{4} \sum\limits_{d = 1}^n {\sum\limits_{i,j,s,t,p = 1}^{2n} {\left( {{T_{spij}} + \sqrt { - 1} {T_{sp'ij}}} \right)\left( {{R_{spdt}}{T_{tdij}} + {R_{spd't}}{T_{td'ij}} + \sqrt { - 1} {R_{spdt}}{T_{td'ij}} - \sqrt { - 1} {R_{spd't}}{T_{tdij}}} \right)} }  \\
   =& \frac{1}{4} \sum\limits_{i,j,s,t,p,q = 1}^{2n} {\left( {{T_{spij}} + \sqrt { - 1} {T_{sp'ij}}} \right)\left( {{R_{spqt}}{T_{tqij}} + \sqrt { - 1} {R_{spqt}}{T_{tq'ij}}} \right)}
   \end{align*}
   \begin{align*}
   =& \frac{1}{4} \sum\limits_{i,j,s,t,p,q = 1}^{2n} {{R_{spqt}}\left( {{T_{spij}} + \sqrt { - 1} {T_{sp'ij}}} \right)\left( {{T_{tqij}} + \sqrt { - 1} {T_{tq'ij}}} \right)}  \\
   =&  - \frac{1}{4}\sum\limits_{i,j,s,t,p,q = 1}^{2n} {{R_{sptq}}{T_{spij}}{T_{tqij}}},
\end{align*}
where we used
$$\sum\limits_{s,p = 1}^{2n} {{R_{spqt}}{T_{sp'ij}}}= -\sum\limits_{s,p = 1}^{2n} {{R_{spqt}}{T_{s'pij}}} = -\sum\limits_{s,p = 1}^{2n} {{R_{psqt}}{T_{ps'ij}}} = 0.$$
For the second term,
\begin{align*}
  &\sum\limits_{B,C,D = 1}^n {\sum\limits_{i,j,k = 1}^{2n} {{R_{\bar DCD\bar B}}{T_{Bijk}}{T_{\bar Cijk}}} }
   = \sum\limits_{B,C = 1}^n {\sum\limits_{i,j,k = 1}^{2n} {\left( {\sum\limits_{m = 1}^{2n} {{R_{mCm\bar B}}}  - \sum\limits_{D = 1}^n {{R_{DC\bar D\bar B}}} } \right){T_{Bijk}}{T_{\bar Cijk}}} }  \\
   =& \sum\limits_{B,C = 1}^n {\sum\limits_{i,j,k = 1}^{2n} {{R_{C\bar B}}{T_{Bijk}}{T_{\bar Cijk}}} }
   = \sum\limits_{C = 1}^n {\sum\limits_{i,j,k = 1}^{2n} {\left( {\sum\limits_{s = 1}^{2n} {{R_{Cs}}{T_{sijk}}}  - \sum\limits_{B = 1}^n {{R_{CB}}{T_{\overline B ijk}}} } \right){T_{\overline C ijk}}} }  \\
   =& \sum\limits_{C = 1}^n {\sum\limits_{i,j,k,s = 1}^{2n} {{R_{Cs}}{T_{sijk}}{T_{\overline C ijk}}} }
   = \frac{1}{2}\sum\limits_{c = 1}^n {\sum\limits_{i,j,k,s = 1}^{2n} {\left( {{R_{cs}} - \sqrt { - 1} {R_{c's}}} \right){T_{sijk}}\left( {{T_{cijk}} + \sqrt { - 1} {T_{c'ijk}}} \right)} }   \\
   =& \frac{1}{2}\sum\limits_{c = 1}^n {\sum\limits_{i,j,k,s = 1}^{2n} {\left( {{R_{cs}}{T_{cijk}} + {R_{c's}}{T_{c'ijk}} - \sqrt { - 1} {R_{c's}}{T_{cijk}} + \sqrt { - 1} {R_{cs}}{T_{c'ijk}}} \right){T_{sijk}}} }  \\
   =& \frac{1}{2}\sum\limits_{i,j,k,s,t = 1}^{2n} {\left( {{R_{ts}}{T_{tijk}} + \sqrt { - 1} {R_{ts}}{T_{t'ijk}}} \right){T_{sijk}}},
   = \frac{1}{2}\sum\limits_{i,j,k,s,t = 1}^{2n} { {{R_{ts}}{T_{tijk}} } {T_{sijk}}}
\end{align*}
since
$$\sum\limits_{i = 1}^{2n} {{R_{ts}}{T_{t'ijk}}{T_{sijk}}}  = \sum\limits_{i = 1}^{2n} {{R_{ts}}{T_{t'i'jk}}{T_{si'jk}}}  =  - \sum\limits_{i = 1}^{2n} {{R_{ts}}{T_{tijk}}{T_{s'ijk}}}  = 0.$$
Combining the two results, we obtain
\begin{align*}
&g\left( {{\mathcal{C} }\left( {{T^{{ \odot ^2}{V^{1,0}}}}} \right),\overline {{T^{{ \odot ^2}{V^{1,0}}}}} } \right)  \\
 =&   - \sum\limits_{i,j,s,t,p,q = 1}^{2n} {{R_{sptq}}{T_{spij}}{T_{tqij}}}  + 2\sum\limits_{i,j,k,s,t = 1}^{2n} {{R_{ts}}{T_{tijk}}{T_{sijk}}}  .
\end{align*}
\end{proof}

Next, we will derive formulas for the trace of the Calabi operator and its powers, which will relate certain terms to the eigenvalues of the Calabi operator.

\begin{proposition}
\label{pro4.2}
Let ${\mathcal{C} }$ be the Calabi curvature operator of a complex $n$-dimensional K$\ddot a$hler manifold $\left(M,g,J\right)$. Then
\begin{align*}
tr\mathcal{C}& = \frac{s}{2}, \ \ \ tr\left( {{\mathcal{C}^2}} \right) = \frac{1}{4}{\left| R \right|^2},\\
 tr{\mathcal{C}^3}& =  \sum\limits_{i,j,k,l,s,t = 1}^{2n} { - \frac{1}{2}{R_{ikjl}}{R_{kslt}}{R_{isjt}} + \frac{1}{8}{R_{ijkl}}{R_{klst}}{R_{ijst}}}. \end{align*}
\end{proposition}
\begin{proof}
From the proof of Theorem ~\ref{theorem1.1}(1), we already have $tr\mathcal{C} = \frac{s}{2}$.

Using the definition of Calabi operator
we compute the trace of ${\mathcal{C}^2}$:
\begin{align*}
  tr{\mathcal{C}^2} =& \frac{1}{4}\sum\limits_{A,B = 1}^n {g\left( {{\mathcal{C}^2}\left( {{Z_A} \odot {Z_B}} \right),\overline {{Z_A}}  \odot \overline {{Z_B}} } \right)}  \\
   =& \frac{1}{4}\sum\limits_{A,B,C,D = 1}^n {g\left( {\mathcal{C}\left( {{R_{A\overline C \overline D B}}{Z_C} \odot {Z_D}} \right),\overline {{Z_A}}  \odot \overline {{Z_B}} } \right)}  \\
   =& \frac{1}{4}\sum\limits_{A, \cdots ,F = 1}^n {g\left( {{R_{A\overline C \overline D B}}\left( {{R_{C\overline E \overline F D}}{Z_E} \odot {Z_F}} \right),\overline {{Z_A}}  \odot \overline {{Z_B}} } \right)}   \\
   =& \sum\limits_{A,B,C,D = 1}^n {{R_{A\overline C \overline D B}}{R_{C\overline A \overline B D}}}
   = \sum\limits_{i,j = 1}^{2n} {\sum\limits_{C,D = 1}^n {{R_{i\overline C \overline D j}}{R_{CijD}}} }   \\
   =& \frac{1}{4}\sum\limits_{i,j = 1}^{2n} {\sum\limits_{c,d = 1}^n {\left( {{R_{icdj}} - {R_{ic'd'j}} + \sqrt { - 1} {R_{ic'dj}} + \sqrt { - 1} {R_{icd'j}}} \right)}}\\
    &{{\left( {{R_{icdj}} - {R_{ic'd'j}} - \sqrt { - 1} {R_{ic'dj}} - \sqrt { - 1} {R_{icd'j}}} \right)} }  \hfill \\
   =& \frac{1}{4}\sum\limits_{i,j = 1}^{2n} {\sum\limits_{c,d = 1}^n {\left( {{{\left( {{R_{icdj}} - {R_{ic'd'j}}} \right)}^2} + {{\left( {{R_{ic'dj}} + {R_{icd'j}}} \right)}^2}} \right)} }   \\
   =& \frac{1}{4}\sum\limits_{i,j = 1}^{2n} {\sum\limits_{c,d = 1}^n {\left( {{R_{icdj}}^2 + {R_{ic'd'j}}^2 + {R_{ic'dj}}^2 + {R_{icd'j}}^2 - 2{R_{icdj}}{R_{ic'd'j}} + 2{R_{ic'dj}}{R_{icd'j}}} \right)} }   \\
   =& \frac{1}{4}\sum\limits_{i,j,k,l = 1}^{2n} {{R_{ijkl}}^2  } + \frac{1}{2} \sum\limits_{i,j = 1}^{2n} {\sum\limits_{c,d = 1}^n {\left( { - {R_{icdj}}{R_{ic'd'j}} + {R_{ic'dj}}{R_{icd'j}}} \right)} }.
\end{align*}
Now
\begin{align*}
   - \sum\limits_{i,j = 1}^{2n} {\sum\limits_{c,d = 1}^n {{R_{icdj}}{R_{ic'd'j}}} }  =&  - \sum\limits_{i,j = 1}^{2n} {\sum\limits_{c,d = 1}^n {{R_{icdj'}}{R_{ic'd'j'}}} }  = \sum\limits_{i,j = 1}^{2n} {\sum\limits_{c,d = 1}^n {{R_{icd'j}}{R_{ic'dj}}} }   \\
   =& \sum\limits_{i,j = 1}^{2n} {\sum\limits_{c,d = 1}^n {{R_{icd'j}}\left( {{R_{idc'j}} + {R_{ijdc'}}} \right)} }  \\
   =& \sum\limits_{i,j = 1}^{2n} {\sum\limits_{c,d = 1}^n {\left( {{R_{icdj}}{R_{idcj}} - {R_{icd'j}}{R_{ijd'c}}} \right)} }   \\
   =& \sum\limits_{i,j = 1}^{2n} {\sum\limits_{c,d = 1}^n {\left( {{R_{icdj}}{R_{idcj}} - \left( {{R_{id'cj}} + {R_{ijd'c}}} \right){R_{ijd'c}}} \right)} }   \\
   =& \sum\limits_{i,j = 1}^{2n} {\sum\limits_{c,d = 1}^n {\left( {{R_{icdj}}{R_{idcj}} - \frac{1}{2}{R_{ijd'c}}{R_{ijd'c}}} \right)} }   \\
   =& \sum\limits_{i,j = 1}^{2n} {\sum\limits_{c,d = 1}^n {\left( {\left( {{R_{icdj}}{R_{icdj}} - \frac{1}{2}{R_{ijdc}}{R_{ijdc}}} \right) - \frac{1}{2}{R_{ijd'c}}{R_{ijd'c}}} \right)} }
   \end{align*}
   \begin{align*}
   =& \sum\limits_{i,j,l = 1}^{2n} {\sum\limits_{c = 1}^n {\left( {\frac{1}{2}{R_{iclj}}{R_{iclj}} - \frac{1}{2}{R_{ijlc}}{R_{ijlc}}} \right)} }  \\
   =& 0 .
\end{align*}
Hence
$tr\left( {{\mathcal{C}^2}} \right) = \frac{1}{4}{\left| R \right|^2}$.

For the third power we proceed analogously:
\begin{align*}
  tr{\mathcal{C}^3} =& \frac{1}{4}\sum\limits_{A,B = 1}^n {g\left( {{\mathcal{C}^3}\left( {{Z_A} \odot {Z_B}} \right),\overline {{Z_A}}  \odot \overline {{Z_B}} } \right)}\\
   =& \frac{1}{4}\sum\limits_{A,B,C,D = 1}^n {g\left( {{\mathcal{C}^2}\left( {{R_{A\overline C \overline D B}}{Z_C} \odot {Z_D}} \right),\overline {{Z_A}}  \odot \overline {{Z_B}} } \right)}  \\
   =& \frac{1}{4}\sum\limits_{A, \cdots ,F = 1}^n {g\left( {{R_{A\overline C \overline D B}}\mathcal{C}\left( {{R_{C\overline E \overline F D}}{Z_E} \odot {Z_F}} \right),\overline {{Z_A}}  \odot \overline {{Z_B}} } \right)}\\
   =& \sum\limits_{A,B,C,D,E,F = 1}^n {{R_{A\overline C \overline D B}}{R_{C\overline E \overline F D}}{R_{E\overline A \overline B F}}}   \\
   =& \sum\limits_{i,j,k,l = 1}^{2n} {\sum\limits_{C,D = 1}^n {{R_{iklj}}{R_{k\overline C \overline D l}}{R_{iCDj}}} }\\
   =& \frac{1}{4}\sum\limits_{i,j,k,l = 1}^{2n} {\sum\limits_{c,d = 1}^n {{R_{iklj}}\left( {{R_{kcdl}} - {R_{kc'd'l}} + \sqrt { - 1} {R_{kc'dl}} + \sqrt { - 1} {R_{kcd'l}}} \right)}}\\
    &\times{{\left( {{R_{icdj}} - {R_{ic'd'j}} - \sqrt { - 1} {R_{ic'dj}} - \sqrt { - 1} {R_{icd'j}}} \right)} }   \\
   =& \frac{1}{4} \sum\limits_{i,j,k,l = 1}^{2n} {\sum\limits_{c,d = 1}^n {{R_{iklj}}\left( {\left( {{R_{kcdl}} - {R_{kc'd'l}}} \right)\left( {{R_{icdj}} - {R_{ic'd'j}}} \right) + \left( {{R_{kc'dl}} + {R_{kcd'l}}} \right)\left( {{R_{ic'dj}} + {R_{icd'j}}} \right)} \right.} }   \\
  &\left. { - \sqrt { - 1} \left( {{R_{kcdl}} - {R_{kc'd'l}}} \right)\left( {{R_{ic'dj}} + {R_{icd'j}}} \right) + \sqrt { - 1} \left( {{R_{icdj}} - {R_{ic'd'j}}} \right)\left( {{R_{kc'dl}} + {R_{kcd'l}}} \right)} \right) ,
\end{align*}
the last two terms
$${\sqrt { - 1} }{{R_{iklj}}\left( - {\left( {{R_{kcdl}} - {R_{kc'd'l}}} \right)\left( {{R_{ic'dj}} + {R_{icd'j}}} \right) +  \left( {{R_{icdj}} - {R_{ic'd'j}}} \right)\left( {{R_{kc'dl}} + {R_{kcd'l}}} \right)} \right)}$$
vanish because the pairs
 $\left( {i,j} \right)$ and $\left( {k,l} \right)$ are symmetric in the first factor but antisymmetric in the second. Therefore
\begin{align*}
   tr{\mathcal{C}^3} =& \frac{1}{4} \sum\limits_{i,j = 1}^{2n} {\sum\limits_{c,d = 1}^n {{R_{iklj}}\left( {{R_{kcdl}}{R_{icdj}} - {R_{kcdl}}{R_{ic'd'j}} - {R_{kc'd'l}}{R_{icdj}} + {R_{kc'd'l}}{R_{ic'd'j}}} \right.} }  \hfill \\
    &\left. { + {R_{kc'dl}}{R_{ic'dj}} + {R_{kc'dl}}{R_{icd'j}} + {R_{kcd'l}}{R_{ic'dj}} + {R_{kcd'l}}{R_{icd'j}}} \right) \hfill \\
   =& \frac{1}{4} \sum\limits_{i,j,k,l,s,t = 1}^{2n} {{R_{iklj}}{R_{kstl}}{R_{istj}}}\\
    & + \frac{1}{4} \sum\limits_{i,j,k,l = 1}^{2n} {\sum\limits_{c,d = 1}^n {{R_{iklj}}\left( { - {R_{kcdl}}{R_{ic'd'j}} - {R_{kc'd'l}}{R_{icdj}} + {R_{kc'dl}}{R_{icd'j}} + {R_{kcd'l}}{R_{ic'dj}}} \right)} }.
\end{align*}
Now
\begin{align*}
   &- \sum\limits_{i,j,k,l = 1}^{2n} {\sum\limits_{c,d = 1}^n {{R_{iklj}}{R_{kc'd'l}}{R_{icdj}}} }  = \sum\limits_{i,j,k,l = 1}^{2n} {\sum\limits_{c,d = 1}^n {{R_{iklj}}{R_{kc'dl}}{R_{icd'j}}} }  \hfill \\
   =&  - \sum\limits_{i,j,k,l = 1}^{2n} {\sum\limits_{c,d = 1}^n {{R_{iklj}}{R_{kcdl}}{R_{ic'd'j}}} }  = \sum\limits_{i,j,k,l = 1}^{2n} {\sum\limits_{c,d = 1}^n {{R_{iklj}}{R_{kcd'l}}{R_{ic'dj}}} }  \hfill \\
   =& \sum\limits_{i,j,k,l = 1}^{2n} {\sum\limits_{c,d = 1}^n {{R_{iklj}}{R_{kcd'l}}\left( {{R_{idc'j}} + {R_{ijdc'}}} \right)} }  \hfill \\
   =& \sum\limits_{i,j,k,l = 1}^{2n} {\sum\limits_{c,d = 1}^n {{R_{iklj}}{R_{kcdl}}{R_{idcj}} - {R_{iklj}}{R_{kcd'l}}{R_{ijd'c}}} }  \hfill \\
   =& \sum\limits_{i,j,k,l = 1}^{2n} {\sum\limits_{c,d = 1}^n {{R_{iklj}}{R_{kcdl}}\left( {{R_{icdj}} + {R_{ijcd}}} \right) - {R_{iklj}}\left( {{R_{kd'cl}} + {R_{kld'c}}} \right){R_{ijd'c}}} }  \hfill \\
   =& \sum\limits_{i,j,k,l = 1}^{2n} {\sum\limits_{c,d = 1}^n {{R_{iklj}}{R_{kcdl}}{R_{icdj}} - {R_{iklj}}{R_{kcdl}}{R_{ijdc}} - \frac{1}{2}{R_{iklj}}{R_{kld'c}}{R_{ijd'c}}} }  \hfill \\
   =& \sum\limits_{i,j,k,l = 1}^{2n} {\sum\limits_{c,d = 1}^n {{R_{iklj}}{R_{kcdl}}{R_{icdj}} - \frac{1}{2}{R_{iklj}}{R_{kldc}}{R_{ijdc}} - \frac{1}{2}{R_{iklj}}{R_{kld'c}}{R_{ijd'c}}} }  \hfill \\
   =& \sum\limits_{i,j,k,l,s,t = 1}^{2n} { - \frac{1}{4}{R_{ikjl}}{R_{kslt}}{R_{isjt}} + \frac{1}{4}{R_{ikjl}}{R_{klst}}{R_{ijst}}}  \hfill \\
   =& \sum\limits_{i,j,k,l,s,t = 1}^{2n} { - \frac{1}{4}{R_{ikjl}}{R_{kslt}}{R_{isjt}} + \frac{1}{8}{R_{ijkl}}{R_{klst}}{R_{ijst}}} .
\end{align*}
Consequently
$$ tr{\mathcal{C}^3} =  \sum\limits_{i,j,k,l,s,t = 1}^{2n} { - \frac{1}{2}{R_{ikjl}}{R_{kslt}}{R_{isjt}} + \frac{1}{8}{R_{ijkl}}{R_{klst}}{R_{ijst}}}. $$
\end{proof}

\begin{lemma}
\label{lemma4.3}
Let $T$ be a K$\ddot a$hler algebraic curvature tensor, and let $\left\{ {{S^\alpha }} \right\}_{\alpha  = 1}^N$ be an orthonormal basis of ${ \odot ^2}{V^{1,0}}$. Then
 $$ \sum\limits_{\alpha  = 1}^N {{{\left| {{S^\alpha }T} \right|}^2}}=n{\left| T \right|^2} - 2{\left| {Ric^T} \right|^2} $$
and
$${\left| {S^{\alpha}T} \right|^2} \le 4\left( {\frac{1}{2}{{\left| T \right|}^2} - \frac{1}{n}{{\left| {Ri{c^T}} \right|}^2}} \right). $$
\end{lemma}
\begin{proof}
\begin{align*}
  \sum\limits_{\alpha  = 1}^N {{{\left| {{S^\alpha }T} \right|}^2}}  =& \frac{1}{4}\sum\limits_{A,B = 1}^n {g\left( {\left( {{Z_A} \odot {Z_B}} \right)T,\left( {\overline {{Z_A}}  \odot \overline {{Z_B}} } \right)T} \right)}  \hfill \\
   =& 2\sum\limits_{A,B = 1}^n {\sum\limits_{i,j = 1}^{2n} {\left( {{T_{B\overline A ij}}{T_{A\overline B ij}} + {T_{B\overline B ij}}{T_{A\overline A ij}}} \right)} }  + 2\left( {n + 1} \right)\sum\limits_{A = 1}^n {\sum\limits_{i,j,k = 1}^{2n} {{T_{Aijk}}{T_{\overline A ijk}}} }  \hfill \\
   =& 2\left( { - \frac{1}{2}{{\left| T \right|}^2} - {{\left| {Ric^T } \right|}^2}} \right) + 2\left( {n + 1} \right)\left( {\frac{1}{2}{{\left| T \right|}^2}} \right) \hfill \\
   =& n{\left| T \right|^2} - 2{\left| {Ric^T} \right|^2} .
\end{align*}
Given a unit vector $S \in { \odot ^2}{V^{1,0}}$, choose an orthonormal basis $\left\{ {{Z_A}} \right\}_{A = 1}^n$ of $V^{1,0}$ such that
$$S = \sum\limits_{A = 1}^n {{\rho _A}{Z_A} \otimes {Z_A}},\ \ {\left| S \right|^2} = \sum\limits_{A = 1}^n {\rho _A^2}.$$ Thus
\begin{align*}
  {\left| {ST} \right|^2} =& \sum\limits_{i,j,k,l = 1}^{2n} {{{\left( {ST} \right)}_{ijkl}}\overline {{{\left( {ST} \right)}_{ijkl}}} }  \hfill \\
   =& \sum\limits_{i,j,k,l = 1}^{2n} {{{\left| {\sum\limits_{A = 1}^n {{\rho _A}\left( {T\left( {{Z_A} \otimes {Z_A}\left( {{e_i}} \right),{e_j},{e_k},{e_l}} \right) +  \cdots  + T\left( {{e_i},{e_j},{e_k},{Z_A} \otimes {Z_A}\left( {{e_l}} \right)} \right)} \right)} } \right|}^2}}  \hfill \\
   \le& \sum\limits_{i,j,k,l = 1}^{2n} {\left( {\sum\limits_{B = 1}^n {{\rho _B}^2} \sum\limits_{A = 1}^n {{{\left| {\left( {T\left( {{Z_A} \otimes {Z_A}\left( {{e_i}} \right),{e_j},{e_k},{e_l}} \right) +  \cdots  + T\left( {{e_i},{e_j},{e_k},{Z_A} \otimes {Z_A}\left( {{e_l}} \right)} \right)} \right)} \right|}^2}} } \right)}  \hfill \\
   =& 4{\left| S \right|^2}\left( {\sum\limits_{j,k,l = 1}^{2n} {\sum\limits_{A,E = 1}^n {T\left( {{Z_A} \otimes {Z_A}\left( {\overline {{Z_E}} } \right),{e_j},{e_k},{e_l}} \right)T\left( {\overline {{Z_A}}  \otimes \overline {{Z_A}} \left( {{Z_E}} \right),{e_j},{e_k},{e_l}} \right)} } } \right. \hfill \\
   &+ \sum\limits_{k,l = 1}^{2n} {\sum\limits_{A,E,F = 1}^n {T\left( {{Z_A} \otimes {Z_A}\left( {\overline {{Z_E}} } \right),\overline {{Z_F}} ,{e_k},{e_l}} \right)T\left( {{Z_E},\overline {{Z_A}}  \otimes \overline {{Z_A}} \left( {{Z_F}} \right),{e_k},{e_l}} \right)} }  \hfill \\
  &\left. { + 2\sum\limits_{j,l = 1}^{2n} {\sum\limits_{A,E,F = 1}^n {T\left( {{Z_A} \otimes {Z_A}\left( {\overline {{Z_E}} } \right),{e_j},\overline {{Z_F}} ,{e_l}} \right)T\left( {{Z_E},{e_j},\overline {{Z_A}}  \otimes \overline {{Z_A}} \left( {{Z_F}} \right),{e_l}} \right)} } } \right) \hfill \\
   =& 4{\left| S \right|^2}\left( {\sum\limits_{j,k,l = 1}^{2n} {\sum\limits_{A = 1}^n {{T_{Ajkl}}{T_{\overline A jkl}}}  + \sum\limits_{k,l = 1}^{2n} {\sum\limits_{A = 1}^n {\left( {{T_{A\overline A kl}}{T_{A\overline A kl}} + 2{T_{Ak\overline A l}}{T_{Ak\overline A l}}} \right)} } } } \right)\\
   \le& 4{\left| S \right|^2}\left( {\frac{1}{2}{{\left| T \right|}^2} - \frac{1}{n}{{\left| {Ri{c^T}} \right|}^2}} \right),
\end{align*}
 because
\begin{align*}
  \sum\limits_{k,l = 1}^{2n} {\sum\limits_{A = 1}^n {{T_{A\overline A kl}}{T_{A\overline A kl}}} }  =&  - \sum\limits_{k,l = 1}^{2n} {\sum\limits_{a = 1}^n {{T_{aa'kl}}{T_{aa'kl}}} }  \hfill \\
   \le&  - \frac{1}{n}\sum\limits_{k,l = 1}^{2n} {{{\left( {\sum\limits_{a = 1}^n {{T_{aa'kl}}} } \right)}^2}}  \hfill \\
   =&  - \frac{1}{n}\sum\limits_{k,l = 1}^{2n} {{{\left( {Ri{c^T}\left( {{e_k},J{e_l}} \right)} \right)}^2}}  \hfill \\
   =&  - \frac{1}{n}{\left| {Ri{c^T}} \right|^2} .
\end{align*}
\end{proof}

For the standard complex projective space $\left( \mathbb{CP}^n,g_{FS} \right)$ with the Fubini-Study metric, the components of the Riemann, Ricci and scalar curvatures are
$${\left( {{R_{\mathbb{C}{\mathbb{P}^n}}}} \right)_{ijkl}} = {g_{ik}}{g_{jl}} - {g_{il}}{g_{jk}} + {g_{i'k}}{g_{j'l}} - {g_{i'l}}{g_{j'k}} + 2{g_{i'j}}{g_{k'l}},$$
$${\left( {Ric_{\mathbb{C}{\mathbb{P}^n}}} \right)_{ik}} = 2\left( {n + 1} \right){g_{ik}}$$
and
$$s_{\mathbb{C}{\mathbb{P}^n}} = 4n\left( {n + 1} \right).$$
\begin{proposition}
\label{proposition4.4}
Let $R$ be the Riemann curvature tensor of an $n$-dimensional K$\ddot a$hler manifold. Define $${R^m} = R - {\frac{{ms}}{{4n\left( {n + 1} \right)}}}R_{\mathbb{CP}^n}.$$ Then
$${\left| R \right|^2} = {\left| {{R^m}} \right|^2} - \frac{{2m\left( {m - 2} \right)}}{{n\left( {n + 1} \right)}}{s^2},\ \
{\left| {Ric} \right|^2} = {\left| {Ri{c^m}} \right|^2} - \frac{{m\left( {m - 2} \right)}}{{2n}}{s^2}.$$
\end{proposition}
\begin{proof}
A direct computation gives
$${\left| R \right|^2} = {\left| {{R^m} + \frac{{ms}}{{4n\left( {n + 1} \right)}}{R_{\mathbb{C}{\mathbb{P}^n}}}} \right|^2} = {\left| {{R^m}} \right|^2} + \frac{{ms}}{{2n\left( {n + 1} \right)}}\left\langle {R,{R_{\mathbb{C}{\mathbb{P}^n}}}} \right\rangle  - \frac{{{m^2}{s^2}}}{{16{n^2}{{\left( {n + 1} \right)}^2}}}{\left| {{R_{\mathbb{C}{\mathbb{P}^n}}}} \right|^2}.$$
Now
\begin{align*}
  \left\langle {R,{R_{\mathbb{C}{\mathbb{P}^n}}}} \right\rangle  =& \sum\limits_{i,j,k,l = 1}^{2n} {{R_{ijkl}}\left( {{g_{ik}}{g_{jl}} - {g_{il}}{g_{jk}} + {g_{i'k}}{g_{j'l}} - {g_{i'l}}{g_{j'k}} + 2{g_{i'j}}{g_{k'l}}} \right)}   \\
   =& 2s + 2\sum\limits_{i,j = 1}^{2n} {{R_{iji'j'}}}  + 2\sum\limits_{i,k = 1}^{2n} {{R_{ii'kk'}}}   \\
   =& 8s
\end{align*}
and
 $${\left| {{R_{\mathbb{C}{\mathbb{P}^n}}}} \right|^2} = 8{s_{\mathbb{C}{\mathbb{P}^n}}} = 32n\left( {n + 1} \right).$$
Substituting these values yields
$${\left| R \right|^2} = {\left| {{R^m}} \right|^2} + \frac{{4ms}}{{n\left( {n + 1} \right)}} - \frac{{2{m^2}{s^2}}}{{n\left( {n + 1} \right)}} = {\left| {{R^m}} \right|^2} - \frac{{2m\left( {m - 2} \right)}}{{n\left( {n + 1} \right)}}{s^2}.$$
Similarly,
$${\left| {Ric} \right|^2} = {\left| {Ri{c^m} + \frac{{ms}}{{2n}}g} \right|^2} = {\left| {Ri{c^m}} \right|^2} + \frac{{ms}}{n}\left\langle {Ric,g} \right\rangle  - \frac{{{m^2}{s^2}}}{{2n}} = {\left| {Ri{c^m}} \right|^2} - \frac{{m\left( {m - 2} \right)}}{{2n}}{s^2}.$$
\end{proof}
\begin{lemma}
\label{lemma4.5}
Let $\left(M,g,J\right)$ be an $n$-dimensional K$\ddot a$hler manifold, $\mathcal{C}$ its  Calabi curvature operator, and let
$${\lambda _1}\le{\lambda _2}\le \cdots \le {\lambda _N} \in \mathbb{R}\ \ \ (N = \frac{{n(n + 1)}}{2} = dim_{\mathbb{C}}\left( { \odot ^2}{V^{1,0}} \right))$$ be the eigenvalues of $\mathcal{C}$. Denote by ${{S^\alpha }} \in {{ \odot ^2}{V^{1,0}}}$ a unit eigenvector belonging to ${\lambda _\alpha}$. Then
$$\sum\limits_{\alpha  = 1}^N {{{\left| {{S^\alpha }R} \right|}^2}}  = \sum\limits_{\alpha  = 1}^N {{{\left| {{S^\alpha }{R^m}} \right|}^2}}  - m\left( {m - 2} \right)\frac{{n - 1}}{{n\left( {n + 1} \right)}}{s^2}$$
and
$$\sum\limits_{\alpha  = 1}^N {{\lambda _\alpha }{{\left| {{S^\alpha }R} \right|}^2}}  = \sum\limits_{\alpha  = 1}^N {{\lambda _\alpha }{{\left| {{S^\alpha }{R^m}} \right|}^2}}  + \frac{{4ms}}{{n\left( {n + 1} \right)}}\left( {{{\left| {Ric} \right|}^2} - 2\sum\limits_{\alpha  = 1}^N {{\lambda _\alpha }^2} } \right) - \frac{(n - 1){m^2}}{{{n^2}{{\left( {n + 1} \right)}^2}}}{s^3}.$$
\end{lemma}
\begin{proof}
For the first  identity,  Lemma \ref{lemma4.3} and Proposition \ref{proposition4.4} give
\begin{align*}
  \sum\limits_{\alpha  = 1}^N {{{\left| {{S^\alpha }R} \right|}^2}}  =& n{\left| R \right|^2} - 2{\left| {Ric} \right|^2} \hfill \\
   =& n\left( {{{\left| {{R^m}} \right|}^2} - \frac{{2m\left( {m - 2} \right)}}{{n\left( {n + 1} \right)}}{s^2}} \right) - 2\left( {{{\left| {Ri{c^m}} \right|}^2} - \frac{{m\left( {m - 2} \right)}}{{2n}}{s^2}} \right) \hfill \\
   =& \sum\limits_{\alpha  = 1}^N {{{\left| {{S^\alpha }{R^m}} \right|}^2}}  - \frac{{\left( {n - 1} \right)m\left( {m - 2} \right)}}{{n\left( {n + 1} \right)}}{s^2}.
\end{align*}
For the second identity, write
\begin{align*}
{\left| {{S^\alpha }R} \right|^2} =& {\text{ }}{\left| {{S^\alpha }\left( {{R^m} + \frac{{ms}}{{4n\left( {n + 1} \right)}}{R_{\mathbb{C}{\mathbb{P}^n}}}} \right)} \right|^2}\\
 =& {\left| {{S^\alpha }{R^m}} \right|^2} + \frac{{ms}}{{2n\left( {n + 1} \right)}} \cdot Re \left( g\left( {{S^\alpha }{R_{\mathbb{C}{\mathbb{P}^n}}},\overline {{S^\alpha }} R} \right) \right) - \frac{{{m^2}{s^2}}}{{16{n^2}{{\left( {n + 1} \right)}^2}}}{\left| {{S^\alpha }{R_{\mathbb{C}{\mathbb{P}^n}}}} \right|^2}.
\end{align*}
By  Lemma \ref{lemma2.3}, we have
\begin{align*}
  {\left( {{S^\alpha }{R_{\mathbb{C}{\mathbb{P}^n}}}} \right)_{ijkl}} =& S_{it}^\alpha {R_{\mathbb{C}{\mathbb{P}^n}}}\left( {{e_t},{e_j},{e_k},{e_l}} \right) + S_{jt}^\alpha {R_{\mathbb{C}{\mathbb{P}^n}}}\left( {{e_i},{e_t},{e_k},{e_l}} \right) \hfill \\
   &+ S_{kt}^\alpha {R_{\mathbb{C}{\mathbb{P}^n}}}\left( {{e_i},{e_j},{e_t},{e_l}} \right) + S_{lt}^\alpha {R_{\mathbb{C}{\mathbb{P}^n}}}\left( {{e_i},{e_j},{e_k},{e_t}} \right) \hfill \\
   =& S_{it}^\alpha \left( {{g_{tk}}{g_{jl}} - {g_{tl}}{g_{jk}} + {g_{t'k}}{g_{j'l}} - {g_{t'l}}{g_{j'k}} + 2{g_{t'j}}{g_{k'l}}} \right) \hfill \\
   &+ S_{jt}^\alpha \left( {{g_{ik}}{g_{tl}} - {g_{il}}{g_{tk}} + {g_{i'k}}{g_{t'l}} - {g_{i'l}}{g_{t'k}} + 2{g_{i't}}{g_{k'l}}} \right) \hfill \\
   &+ S_{kt}^\alpha \left( {{g_{it}}{g_{jl}} - {g_{il}}{g_{jt}} + {g_{i't}}{g_{j'l}} - {g_{i'l}}{g_{j't}} + 2{g_{i'j}}{g_{t'l}}} \right) \hfill \\
   &+ S_{lt}^\alpha \left( {{g_{ik}}{g_{jt}} - {g_{it}}{g_{jk}} + {g_{i'k}}{g_{j't}} - {g_{i't}}{g_{j'k}} + 2{g_{i'j}}{g_{k't}}} \right) \hfill \\
   =& 2\left( {S_{ik}^\alpha {g_{jl}} + {g_{ik}}S_{jl}^\alpha  - S_{il}^\alpha {g_{jk}} - {g_{il}}S_{jk}^\alpha } \right).
\end{align*}
Hence
\begin{align*}
   {{\left| {{S^\alpha }{R_{\mathbb{C}{\mathbb{P}^n}}}} \right|}^2}  =& 4\left( {S_{ik}^\alpha {g_{jl}} + {g_{ik}}S_{jl}^\alpha  - S_{il}^\alpha {g_{jk}} - {g_{il}}S_{jk}^\alpha } \right)^2 \\
 =&4\left( {4\left( {2n - 2} \right){{\left| {{S^\alpha }} \right|}^2} + 4{{\left( {tr{S^\alpha }} \right)}^2}} \right)\\
 =& 32\left( {n - 1} \right)
\end{align*}
and
\begin{align*}
  \sum\limits_{\alpha  = 1}^N {{\lambda _\alpha }g\left( {{S^\alpha }{R_{\mathbb{C}{\mathbb{P}^n}}},{\overline {S^\alpha } } R} \right)}  =& 2\sum\limits_{\alpha  = 1}^N {{\lambda _\alpha }{{\left( {S_{ik}^\alpha {g_{jl}} + {g_{ik}}S_{jl}^\alpha  - S_{il}^\alpha {g_{jk}} - {g_{il}}S_{jk}^\alpha } \right)}}{{\left( {{\overline {S^\alpha } } R} \right)}_{ijkl}}}  \hfill \\
   =& 8\sum\limits_{\alpha  = 1}^N {{\lambda _\alpha }\left( {S_{ik}^\alpha {g_{jl}} + {g_{ik}}S_{jl}^\alpha  - S_{il}^\alpha {g_{jk}} - {g_{il}}S_{jk}^\alpha } \right){\overline {S_{ti}^\alpha } } {R_{tjkl}}}  \hfill \\
   =& 16\sum\limits_{\alpha  = 1}^N {{\lambda _\alpha }\left( {{R_{tk}}S_{ik}^\alpha {\overline {S_{tk}^\alpha } }  + {R_{tjkl}}S_{jl}^\alpha {\overline {S_{tk}^\alpha } } } \right)} \\
   =&8{\left| {Ric} \right|^2} - 16\sum\limits_{\alpha  = 1}^N {{\lambda _\alpha }^2} .
\end{align*}
Substituting these results, we get
\begin{align*}
  \sum\limits_{\alpha  = 1}^N {{\lambda _\alpha }{{\left| {{S^\alpha }R} \right|}^2}}  =& \sum\limits_{\alpha  = 1}^N {{\lambda _\alpha }\left( {{{\left| {{S^\alpha }{R^m}} \right|}^2} + \frac{{ms}}{{2n\left( {n + 1} \right)}}Reg\left( {{S^\alpha }{R_{\mathbb{C}{\mathbb{P}^n}}},\overline {{S^\alpha }} R} \right) - \frac{{{m^2}{s^2}}}{{16{n^2}{{\left( {n + 1} \right)}^2}}}{{\left| {{S^\alpha }{R_{\mathbb{C}{\mathbb{P}^n}}}} \right|}^2}} \right)}  \hfill \\
   =& \sum\limits_{\alpha  = 1}^N {{\lambda _\alpha }{{\left| {{S^\alpha }{R^m}} \right|}^2}}  + \frac{{ms}}{{2n\left( {n + 1} \right)}}\left( {8{{\left| {Ric} \right|}^2} - 16\sum\limits_{\alpha  = 1}^N {{\lambda _\alpha }^2} } \right) - \frac{{2\left( {n - 1} \right){m^2}{s^2}}}{{{n^2}{{\left( {n + 1} \right)}^2}}}\sum\limits_{\alpha  = 1}^N {{\lambda _\alpha }}  \hfill \\
   =& \sum\limits_{\alpha  = 1}^N {{\lambda _\alpha }{{\left| {{S^\alpha }{R^m}} \right|}^2}}  + \frac{{4ms}}{{n\left( {n + 1} \right)}}{\left| {Ric} \right|^2} - \frac{{8ms}}{{n\left( {n + 1} \right)}}\sum\limits_{\alpha  = 1}^N {{\lambda _\alpha }^2}  - \frac{{\left( {n - 1} \right){m^2}}}{{{n^2}{{\left( {n + 1} \right)}^2}}}{s^3}. \hfill \\
\end{align*}
\end{proof}
We can now derive a Bochner-type formula involving the eigenvalues of the Calabi operator.
\begin{lemma}
Under the same hypotheses as Lemma~\ref{lemma4.5}, and additionally, g is an Einstein metric here. For $\beta \le \frac{n}{2}$, if the smallest $\beta$ eigenvalues satisfy $${\lambda _1} +  \cdots  + {\lambda _\beta} \ge  - \beta\theta \bar \lambda, $$ where $\bar \lambda$ is  the average of all eigenvalues, then
\begin{equation}
\label{equation4.1}
\begin{aligned}
  \left\langle {\Delta R,R} \right\rangle  =& 2\sum\limits_{\alpha  = 1}^N {{\lambda _\alpha }{{\left| {{S^\alpha }{R^1}} \right|}^2}}  + 8\sum\limits_{\alpha  = 1}^N {{\lambda _\alpha }^3}  - 4\left( {n + 5} \right)\bar \lambda \sum\limits_{\alpha  = 1}^N {{\lambda _\alpha }^2}  + 2n\left( {n + 1} \right)\left( {n + 3} \right){{\bar \lambda }^3} \hfill \\
   \ge& 8\sum\limits_{\alpha  = 1}^N {{\lambda _\alpha }^3}  - \left( {8n\theta  + 4\left( {n + 5} \right)} \right)\bar \lambda \sum\limits_{\alpha  = 1}^N {{\lambda _\alpha }^2}  + \left( {4{n^2}\left( {n + 1} \right)\theta  + 2n\left( {n + 1} \right)\left( {n + 3} \right)} \right){{\bar \lambda }^3} .
\end{aligned}
\end{equation}
\end{lemma}
\begin{proof}
For an Einstein metric, it follows from Proposition 3.1 in \cite{dai2024einstein} that
$$\left\langle {\Delta R,R} \right\rangle  = \frac{s}{n}{\left| R \right|^2}  - \sum\limits_{i,j,k,l,s,t = 1}^{2n} {{R_{ijkl}}{R_{ijst}}{R_{klst}}}  - 4\sum\limits_{i,j,k,l,s,t = 1}^{2n} {{R_{ijkl}}{R_{iskt}}{R_{jslt}}}.$$
Using Proposition \ref{pro4.1} and Proposition \ref{pro4.2}, this becomes
\begin{align*}
 \left\langle {\Delta R,R} \right\rangle  =& 2g\left( {\mathcal{C}\left( {{R^{{ \odot ^2}{V^{1,0}}}}} \right),\overline {{R^{{ \odot ^2}{V^{1,0}}}}} } \right) + 8tr\left( {{\mathcal{C}^3}} \right) - \frac{4s}{n}tr\left( {{\mathcal{C}^2}} \right)\\
 =& 2\sum\limits_{\alpha  = 1}^N {{\lambda _\alpha }{{\left| {{S^\alpha }R} \right|}^2}}  + 8\sum\limits_{\alpha  = 1}^N {{\lambda _\alpha }^3}  - \frac{4}{n}s\sum\limits_{\alpha  = 1}^N {{\lambda _\alpha }^2} .
\end{align*}
Now apply  Lemma ~\ref{lemma4.5} with $m=1$. In the Einstein case  $$s = n\left( {n + 1} \right)\bar \lambda ,{\left| {Ric} \right|^2} = \frac{{{s^2}}}{{2n}} = \frac{{n{{\left( {n + 1} \right)}^2}}}{2}{{\bar \lambda }^2}.$$ Thus
\begin{align*}
  \left\langle {\Delta R,R} \right\rangle
   =& 2\left( {\sum\limits_{\alpha  = 1}^N {{\lambda _\alpha }{{\left| {{S^\alpha }{R^1}} \right|}^2}}  + \frac{{4s}}{{n\left( {n + 1} \right)}}\left( {{{\left| {Ric} \right|}^2} - 2\sum\limits_{\alpha  = 1}^N {{\lambda _\alpha }^2} } \right) - \frac{{n - 1}}{{{n^2}{{\left( {n + 1} \right)}^2}}}{s^3}} \right) \hfill \\
   &+ 8\sum\limits_{\alpha  = 1}^N {{\lambda _\alpha }^3}  - \frac{4}{n}s\sum\limits_{\alpha  = 1}^N {{\lambda _\alpha }^2}  \hfill \\
   =& 2\sum\limits_{\alpha  = 1}^N {{\lambda _\alpha }{{\left| {{S^\alpha }{R^1}} \right|}^2}}  + 8\sum\limits_{\alpha  = 1}^N {{\lambda _\alpha }^3}  - 4\left( {n + 5} \right)\bar \lambda \sum\limits_{\alpha  = 1}^N {{\lambda _\alpha }^2}  + 2n\left( {n + 1} \right)\left( {n + 3} \right){{\bar \lambda }^3}.
\end{align*}
Because  ${\lambda _1} +  \cdots  + {\lambda _\beta} \ge  - \beta\theta \bar \lambda $ and by Lemma ~\ref{lemma4.3},  $$\beta \le \frac{n}{2} \le \frac{{\sum\limits_{\alpha  = 1}^N {{{\left| {{S^\alpha }{R^1}} \right|}^2}} }}{{\mathop {\max }\limits_{1 \le \beta  \le N}  {{{\left| {{S^\beta }{R^1}} \right|}^2}} }},$$  Lemma \ref{lemma2.5} yields
\begin{align*}
  \sum\limits_{\alpha  = 1}^N {{\lambda _\alpha }{{\left| {{S^\alpha }{R^1}} \right|}^2}}  \ge &  - \theta \bar \lambda \sum\limits_{\alpha  = 1}^N {{{\left| {{S^\alpha }{R^1}} \right|}^2}}  \hfill \\
   =&  - \theta \bar \lambda \left( {\sum\limits_{\alpha  = 1}^N {{{\left| {{S^\alpha }R} \right|}^2}}  - \left( {n - 1} \right)n\left( {n + 1} \right){{\bar \lambda }^2}} \right) \hfill \\
   =&  - \theta \bar \lambda \left( {\left( {n{{\left| R \right|}^2} - 2{{\left| {Ric} \right|}^2}} \right) - \left( {n - 1} \right)n\left( {n + 1} \right){{\bar \lambda }^2}} \right) \hfill \\
   =&  - 4n\theta \bar \lambda \sum\limits_{\alpha  = 1}^N {{\lambda _\alpha }^2}  + 2{n^2}\left( {n + 1} \right)\theta {{\bar \lambda }^3}.
\end{align*}
Substituting this estimate we obtain
$$\left\langle {\Delta R,R} \right\rangle
\ge 8\sum\limits_{\alpha  = 1}^N {{\lambda _\alpha }^3}  - \left( {8n\theta  + 4\left( {n + 5} \right)} \right)\bar \lambda \sum\limits_{\alpha  = 1}^N {{\lambda _\alpha }^2}  + \left( {4{n^2}\left( {n + 1} \right)\theta  + 2n\left( {n + 1} \right)\left( {n + 3} \right)} \right){{\bar \lambda }^3} .$$
\end{proof}

\begin{remark}
\label{remark4.7}
In fact, the K{\"a}hler algebraic curvature tensor $R^1$ is precisely the Bochner tensor of the K{\"a}hler Einstein manifold $\left( M,g,J \right)$. In formula ~\eqref{equation4.1}, we have
\begin{align*}
&8\sum\limits_{\alpha  = 1}^N {{\lambda _\alpha }^3}  - 4\left( {n + 5} \right)\bar \lambda \sum\limits_{\alpha  = 1}^N {{\lambda _\alpha }^2}  + 2n\left( {n + 1} \right)\left( {n + 3} \right){{\bar \lambda }^3}\\
 =& 8\left( {\sum\limits_{\alpha  = 1}^N {{\lambda _\alpha }^3}  - N{{\bar \lambda }^3}} \right) - 4\left( {n + 5} \right)\bar \lambda \left( {\sum\limits_{\alpha  = 1}^N {{\lambda _\alpha }^2}  - N{{\bar \lambda }^2}} \right)
\end{align*}
and
\begin{align*}
&8\sum\limits_{\alpha  = 1}^N {{\lambda _\alpha }^3}  - \left( {8n\theta  + 4\left( {n + 5} \right)} \right)\bar \lambda \sum\limits_{\alpha  = 1}^N {{\lambda _\alpha }^2}  + \left( {4{n^2}\left( {n + 1} \right)\theta  + 2n\left( {n + 1} \right)\left( {n + 3} \right)} \right){{\bar \lambda }^3}\\
 =& 8\left( {\sum\limits_{\alpha  = 1}^N {{\lambda _\alpha }^3}  - N{{\bar \lambda }^3}} \right) - \left( {8n\theta  + 4\left( {n + 5} \right)} \right)\bar \lambda \left( {\sum\limits_{\alpha  = 1}^N {{\lambda _\alpha }^2}  - N{{\bar \lambda }^2}} \right).
\end{align*}
Thus, when all eigenvalues are equal, i.e., ${\lambda _1} =  \cdots  = {\lambda _N} = \bar \lambda $, both expressions vanish.
\end{remark}

\section{The proof of Theorem~\ref{theorem1.4}}
We shall use Lemma 3.3 and Lemma 3.4 of \cite{fu2025new} to estimate  formula~\eqref{equation4.1}.
\begin{lemma}$(${\cite{fu2025new}}$)$
\label{lemma5.1}
Let $0 \le {\lambda _1} \le {\lambda _2} \le  \cdots  \le {\lambda _N}$ and $\sum\limits_{\alpha  = 1}^N {{\lambda _\alpha }}  = C = \frac{{N\left( {N - 1} \right)}}{{2N - 1}}$, then the function
$$F\left( {{\lambda _1},{\lambda _2}, \cdots ,{\lambda _N}} \right) = \sum\limits_{\alpha  = 1}^N {\lambda _\alpha ^3}  - \sum\limits_{\alpha  = 1}^N {\lambda _\alpha ^2} $$
attains its minimal at  the points $\left( {\frac{C}{N},\frac{C}{N}, \cdots ,\frac{C}{N}} \right)$ and $\left( {0,\frac{C}{{N - 1}}, \cdots ,\frac{C}{{N - 1}}} \right)$.
\end{lemma}

\begin{lemma}$(${\cite{fu2025new}}$)$
\label{lemma5.2}
Let ${\lambda _1} \le {\lambda _2} \le  \cdots  \le {\lambda _N}$ satisfy ${\lambda _1} + {\lambda _2} +  \cdots  + {\lambda _\beta} \ge  - \beta\theta \bar \lambda $, where ${\bar \lambda }$ is the average of $\{\lambda _k\}_{k=1}^{N}$. Then
$${\lambda _1} \ge -\frac{{\left( {N - 1} \right)\beta\theta  + N\left( {\beta - 1} \right)}}{{N - \beta}}\bar \lambda. $$
\end{lemma}
Set ${\mu _\alpha } = {\lambda _\alpha } + D\bar \lambda$, $D = \frac{{\left( {N - 1} \right)\beta\theta  + N\left( {\beta - 1} \right)}}{{N - \beta}}$. Then ${\mu _\alpha } \ge 0$ and $\sum\limits_{\alpha  = 1}^N {{\mu _\alpha }}  = \left( {1 + D} \right)N\bar \lambda $. Rewriting \eqref{equation4.1} we obtain
\begin{align*}
\frac{1}{8}\left\langle {\Delta R,R} \right\rangle  \ge & \sum\limits_{\alpha  = 1}^N {{\lambda _\alpha }^3}  - \left( {n\theta  + \frac{{n + 5}}{2}} \right)\bar \lambda \sum\limits_{\alpha  = 1}^N {{\lambda _\alpha }^2}  + \left( {\frac{{{n^2}\left( {n + 1} \right)}}{2}\theta  + \frac{{n\left( {n + 1} \right)\left( {n + 3} \right)}}{4}} \right){{\bar \lambda }^3}\\
 =& \sum\limits_{\alpha  = 1}^N {{{\left( {{\mu _\alpha } - D\bar \lambda } \right)}^3}}  - \left( {n\theta  + \frac{{n + 5}}{2}} \right)\bar \lambda \sum\limits_{\alpha  = 1}^N {{{\left( {{\mu _\alpha } - D\bar \lambda } \right)}^2}} \\
 &+ \left( {\frac{{{n^2}\left( {n + 1} \right)}}{2}\theta  + \frac{{n\left( {n + 1} \right)\left( {n + 3} \right)}}{4}} \right){{\bar \lambda }^3}\\
 =& \sum\limits_{\alpha  = 1}^N {{\mu _\alpha }^3}  - \left( {3D + n\theta  + \frac{{n + 5}}{2}} \right)\bar \lambda \sum\limits_{\alpha  = 1}^N {{\mu _\alpha }^2}  + G\left( {\theta ,n} \right){{\bar \lambda }^3}\\
 =& F\left( {{\mu _1},{\mu _2}, \cdots ,{\mu _N}} \right) + G\left( {\theta ,n} \right){{\bar \lambda }^3},
\end{align*}
where $F$ is  defined on the ordered sequence $0 \le {\mu _1} \le {\mu _2} \le  \cdots  \le {\mu _N}$ and $G\left( {\theta ,n} \right)$ is a constant independent of $\{\mu _\alpha \}$. Choose $\theta$ such that
$$3D + n\theta  + \frac{{n + 5}}{2} = \frac{{2N - 1}}{{N - 1}}\left( {1 + D} \right),$$
that is
$$\theta \left( {n,\beta} \right) = \frac{{ - 2N\beta - \left( {n - 1} \right)N + \left( {n + 3} \right)\beta}}{{2N\beta + 2nN - 2\left( {n + 2} \right)\beta}},$$
where we have used the definition of $D$. Continuing to transform $F$,
\begin{align*}
F =& \sum\limits_{\alpha  = 1}^N {{\mu _\alpha }^3}  - \frac{{2N - 1}}{{N - 1}}\left( {1 + D} \right)\bar \lambda \sum\limits_{\alpha  = 1}^N {{\mu _\alpha }^2} \\
 =& {\left( {\frac{{2N - 1}}{{N - 1}}\left( {1 + D} \right)\bar \lambda } \right)^3}\left( {\sum\limits_{\alpha  = 1}^N {{{\left( {\frac{{N - 1}}{{\left( {2N - 1} \right)\left( {1 + D} \right)\bar \lambda }}{\mu _\alpha }} \right)}^3}}  - \sum\limits_{\alpha  = 1}^N {{{\left( {\frac{{N - 1}}{{\left( {2N - 1} \right)\left( {1 + D} \right)\bar \lambda }}{\mu _\alpha }} \right)}^2}} } \right).
\end{align*}
Note that $$\sum\limits_{\alpha  = 1}^N {\frac{{N - 1}}{{\left( {2N - 1} \right)\left( {1 + D} \right)\bar \lambda }}{\mu _\alpha }}  = \frac{{N - 1}}{{\left( {2N - 1} \right)\left( {1 + D} \right)\bar \lambda }}\sum\limits_{\alpha  = 1}^N {{\mu _\alpha }}  = \frac{{N\left( {N - 1} \right)}}{{2N - 1}},$$
hence the rescaled variables satisfy the hypotheses of Lemma~\ref{lemma5.1}. Consequently the minimum of $F$ is attained at
$$\left( {{\mu _1},{\mu _2}, \cdots ,{\mu _N}} \right) = \left( {\left( {1 + D} \right)\bar \lambda ,\left( {1 + D} \right)\bar \lambda , \cdots ,\left( {1 + D} \right)\bar \lambda } \right)$$
and
$$\left( {{\mu _1},{\mu _2}, \cdots ,{\mu _N}} \right) = \left( {0,\frac{{N\left( {1 + D} \right)}}{{N - 1}}\bar \lambda , \cdots ,\frac{{N\left( {1 + D} \right)}}{{N - 1}}\bar \lambda } \right).$$Returning to the original variables, these two configurations correspond respectively to
$$\left( {{\lambda _1},{\lambda _2}, \cdots ,{\lambda _N}} \right) = \left( {\bar \lambda ,\bar \lambda , \cdots ,\bar \lambda } \right)$$
and
$$\left( {{\lambda _1},{\lambda _2}, \cdots ,{\lambda _N}} \right) = \left( { - D\bar \lambda ,\frac{{N + D}}{{N - 1}}\bar \lambda , \cdots ,\frac{{N + D}}{{N - 1}}\bar \lambda } \right).$$
For both configurations the right-hand side of \eqref{equation4.1} vanishes.

\begin{proof}[{\bf Proof of Theorem~\ref{theorem1.4}}]
From the discussion above we see:  for a K$\ddot a$hler Einstein manifold $\left( M^n,g,J \right)$ whose Calabi operator $\mathcal{C}$ satisfies
 $${\lambda _1} + {\lambda _2} +  \cdots  + {\lambda _\beta} \ge  - \beta\theta \bar \lambda, \ \ \beta \le \frac{n}{2},$$
 with
$$ \theta \left( {n,\beta} \right) = \frac{{ - 2N\beta - \left( {n - 1} \right)N + \left( {n + 3} \right)\beta}}{{2N\beta + 2nN - 2\left( {n + 2} \right)\beta}},$$
we have
$$\left\langle {\Delta R,R} \right\rangle  \ge 0,$$
and the equality holds precisely when the eigenvalue vector is one of the two extremal configurations described above. Now
$$\frac{1}{2}\Delta {\left| R \right|^2} = {\left| {\nabla R} \right|^2} + \left\langle {\Delta R,R} \right\rangle  \ge 0,$$
which implies
$$\nabla R = 0\ \  \text{and} \ \ \left\langle {\Delta R,R} \right\rangle  = 0.$$
Observe that for $\beta >1$, equality in \eqref{equation2.3} holds  only if
$${\lambda _1} =  \cdots  = {\lambda _{\left[ \beta\right]}} = {\lambda _{\left[ \beta \right] + 1}}.$$
Hence, when the right-hand side of \eqref{equation4.1} is zero,
$$\left( {{\lambda _1},{\lambda _2}, \cdots ,{\lambda _N}} \right) = \left( {\bar \lambda ,\bar \lambda , \cdots ,\bar \lambda } \right),$$
that is to say that $\left( M,g,J \right)$ is nonnegative constant holomorphic sectional curvature.

For the borderline case $\beta=1$, we have
$${\lambda _1} \ge  - \frac{{ - 2N - \left( {n - 1} \right)N + \left( {n + 3} \right)}}{{2N + 2nN - 2\left( {n + 2} \right)}}\bar \lambda  = \frac{1}{2}\left( {1 - \frac{1}{{\left( {n + 1} \right)N - \left( {n + 2} \right)}}} \right)\bar \lambda.$$
Thus either ${\lambda _1}=0$, i.e., $\left( M,g,J \right)$ is  locally flat, or the Calabi operator is positive. In the latter case the universal cover $\widetilde M$ of $M$ carries a positive Calabi operator, which implies that $\widetilde M$ is biholomorphic to ${\mathbb{CP}{^n}}$. Since $g$ is Einstein metric,  $\widetilde M$ is in fact isometric to ${\mathbb{CP}{^n}}$ with the Fubini-Study metric.
\end{proof}


\bibliographystyle{amsalpha}
\bibliography{references}

@article{fu2025new,
  title={New rigidity theorem of {E}instein manifolds and curvature operator of the second kind},
  author={Fu, Haiping and Lu, Yao},
  journal={arXiv preprint arXiv:2512.21496},
  year={2025}
}

@article{tachibana_theorem_1974,
	title = {A theorem on {Riemannian} manifolds of positive curvature operator},
	volume = {50},
	doi = {10.3792/PJA/1195518988},
	number = {4},
	journal = {Proceedings of the Japan Academy},
	author = {Tachibana, Shun-ichi },
	month = jan,
	year = {1974},
	pages = {301--302},
}

@article{li2023kahler,
  title={{K}{\"a}hler manifolds and the curvature operator of the second kind},
  author={Li, Xiaolong},
  journal={Mathematische Zeitschrift},
  volume={303},
  number={4},
  pages={101},
  year={2023},
  publisher={Springer}
}

@article{bochner1946vector,
     author = {Bochner, S.},
     title = {Vector fields and {R}icci curvature},
     journal = {Bull. Amer. Math. Soc.},
     volume = {52},
     number = {12},
     year = {1946},
     pages = { 776-797},
     language = {en},
     url = {http://dml.mathdoc.fr/item/1183509635}
}

@article{calabi1960compact,
  title={On compact, locally symmetric {K}{\"a}hler manifolds},
  author={Calabi, Eugenio and Vesentini, Edoardo},
  journal={Annals of Mathematics},
  volume={71},
  number={3},
  pages={472--507},
  year={1960},
  publisher={JSTOR}
}

@article{ogiue1980kaehler,
  title={{K}aehler manifolds of positive curvature operator},
  author={Ogiue, Koichi and Tachibana, Shun-Ichi},
  journal={Proceedings of the American Mathematical Society},
  pages={548--550},
  year={1980},
  publisher={JSTOR}
}

@article{broder2025vanishing,
  title={Vanishing theorems for {H}odge numbers and the {C}alabi curvature operator},
  author={Broder, Kyle and Nienhaus, Jan and Petersen, Peter and Stanfield, James and Wink, Matthias},
  journal={arXiv preprint arXiv:2503.06870},
  year={2025}
}

@article{borel1960curvature,
  title={On the curvature tensor of the {H}ermitian symmetric manifolds},
  author={Borel, Armand},
  journal={Annals of Mathematics},
  volume={71},
  number={3},
  pages={508--521},
  year={1960},
  publisher={JSTOR}
}

@article{sitaramayya1973curvature,
  title={Curvature tensors in {K}{\"a}hler manifolds},
  author={Sitaramayya, Malladi},
  journal={Transactions of the American Mathematical Society},
  volume={183},
  pages={341--353},
  year={1973}
}

@article{ni2018comparison,
  title={Comparison and vanishing theorems for {K}{\"a}hler manifolds},
  author={Ni, Lei and Zheng, Fangyang},
  journal={Calculus of Variations and Partial Differential Equations},
  volume={57},
  number={6},
  pages={151},
  year={2018},
  publisher={Springer}
}

@article{siu1980compact,
  title={Compact {K}{\"a}hler manifolds of positive bisectional curvature},
  author={Siu, Yum-Tong and Yau, Shing-Tung},
  journal={Inventiones mathematicae},
  volume={59},
  number={2},
  pages={189--204},
  year={1980},
  publisher={Springer-Verlag Berlin/Heidelberg}
}

@article{matsushima1957structure,
  title={Sur la structure du groupe d’hom{\'e}omorphismes analytiques d’une certaine vari{\'e}t{\'e} kaehl{\'e}rinne},
  author={Matsushima, Yoz{\^o}},
  journal={Nagoya Mathematical Journal},
  volume={11},
  pages={145--150},
  year={1957},
  publisher={Cambridge University Press}
}

@article{dai2024einstein,
  title={{E}instein manifolds and curvature operator of the second kind},
  author={Dai, Zhi-Lin and Fu, Hai-Ping},
  journal={Calculus of Variations and Partial Differential Equations},
  volume={63},
  number={2},
  pages={53},
  year={2024},
  publisher={Springer}
}

@article{li2025new,
  title={New sphere theorems under curvature operator of the second kind},
  author={Li, Xiaolong},
  journal={Journal of the London Mathematical Society},
  volume={112},
  number={5},
  pages={e70356},
  year={2025},
  publisher={Wiley Online Library}
}

@article{wang_weitzenb$backslash$_2025,
	title = {{W}eitzenb{\"o}ck-{Bochner}-{Kodaira} formulas with quadratic curvature terms},
	journal = {arXiv preprint arXiv:2509.00468},
	author = {Wang, Mingwei  and Yang, Xiaokui },
	year = {2025},
}

@article{petersen2021vanishing,
  title={Vanishing and estimation results for {H}odge numbers},
  author={Petersen, Peter and Wink, Matthias},
  journal={Journal f{\"u}r die reine und angewandte Mathematik (Crelles Journal)},
  volume={2021},
  number={780},
  pages={197--219},
  year={2021},
  publisher={De Gruyter}
}

@article{2021New,
  title={New curvature conditions for the {B}ochner Technique},
  author={ Petersen, Peter  and  Wink, Matthias },
  journal={Inventiones mathematicae},
  volume={224},
  number={2},
  year={2021},
}
\end{document}